\documentclass[final,leqno]{amsart}
\usepackage{amsmath,amssymb,amsfonts,latexsym,stmaryrd}
\usepackage{graphicx,float,epsfig} 
\usepackage{mathrsfs} 
\usepackage{color}

\DeclareGraphicsExtensions{ .eps,.ps} \usepackage{subfig}
\usepackage{multirow}
\usepackage{url}
\usepackage{showkeys}
\usepackage[linesnumbered,algoruled,boxed]{algorithm2e}
\usepackage[colorlinks=false, linkcolor=blue,  citecolor=OliveGreen,  pdfstartview={}{}]{hyperref} 
\usepackage[dvipsnames]{xcolor}

\def\ds{\displaystyle}
\def\CE{\mathcal{E}}

\def\O{\Omega}

\def\g{\gamma}

\def\l{\lambda}

\def\E{K}

\renewcommand\sp{\mathop{\mathrm{Sp}}\nolimits}
\newtheorem{remark}{Remark}[section]
\newtheorem{lemma}{Lemma}[section]
\newtheorem{corollary}{Corollary}[section]
\newtheorem{theorem}{Theorem}[section]
\newtheorem{prop}{Proposition}[section]

\newcommand{\norm}[1]{\lVert#1\rVert}

\newcommand{\n}{\boldsymbol{n}}
\usepackage{booktabs}
\usepackage{float}

\newcommand\bu{\boldsymbol{u}}
\newcommand\bv{\boldsymbol{v}}
\newcommand\bw{\boldsymbol{w}}
\newcommand\bx{\boldsymbol{x}}
\newcommand\by{\boldsymbol{y}}

\newcommand\bn{\boldsymbol{n}}
\newcommand\bt{\boldsymbol{t}}

\newcommand\bg{\boldsymbol{g}}
\def\bp{\mathbf{p}}

\def\hdel{\widehat{\delta}}
\def\CM{\mathcal{X}}
\def\CN{\mathcal{Y}}

\newcommand\bF{\boldsymbol{f}}

\newcommand\bT{\boldsymbol{T}}

\newcommand\0{\mathbf{0}}

\newcommand\bV{\boldsymbol{V}}

\newcommand\Ph{\mathcal{P}_{h}}
\newcommand\AT{\bold{A}}
\newcommand\BT{\bold{B}}

\newcommand\cT{\mathcal{T}}
\def\CT{{\mathcal T}}

\newcommand\bPi{\boldsymbol{\Pi}}

\newcommand\R{\mathbb{R}}

\renewcommand\H{\mathrm{H}}
\newcommand\Q{\mathrm{Q}}
\renewcommand\L{\mathrm{L}}

\renewcommand\O{\Omega}

\renewcommand\div{\mathop{\mathrm{div}}\nolimits}

\renewcommand\sp{\mathop{\mathrm{sp}}\nolimits}

\newcommand\LO{\L^2(\O)}

\newcommand\HsO{\H^s(\O)}

\newcommand\disp{\displaystyle}

\newcommand{\vertiii}[1]{{\left\vert\kern-0.25ex\left\vert\kern-0.25ex\left\vert #1 
    \right\vert\kern-0.25ex\right\vert\kern-0.25ex\right\vert}}

\begin{document}

\title[VEM for the Stokes eigenvalue problem]
{VEM approximation for the  Stokes eigenvalue problem: a priori and a posteriori error analysis}


\author{Dibyendu Adak}
\address{Department of Mechanical Engineering, Indian Institute of Technology Madras, Chennai-600036, India, and GIMNAP-Departamento de Matem\'atica, Universidad del B\'io-B\'io, Casilla 5-C, Concepci\'on, Chile.}
\email{dibyendu.jumath@gmail.com}
\thanks{The first author was partially supported by
ANID-Chile through FONDECYT  Postdoctorado project 3200242.}

\author{Felipe Lepe}
\address{GIMNAP-Departamento de Matem\'atica, Universidad del B\'io - B\'io, Casilla 5-C, Concepci\'on, Chile.}
\email{flepe@ubiobio.cl}
\thanks{The second author was partially supported by  DICREA through project 2120173 GI/C Universidad del B\'io-B\'io and 
ANID-Chile through FONDECYT project 11200529 (Chile).}
\author{Gonzalo Rivera}
\address{Departamento de Ciencias Exactas, Universidad de Los Lagos,
Casilla 933, Osorno, Chile.}
\email{gonzalo.rivera@ulagos.cl}  
\thanks{The third author was partially by Universidad de Los Lagos through project Regular R02/21.}



\begin{abstract}
The present paper proposes an { \it inf-sup} stable {\it divergence free} virtual element method and associated {\it a priori}, and {\it a posteriori} error analysis to approximate the eigenvalues and eigenfunctions
of the Stokes spectral problem in one shot. For the a priori analysis, we take advantage of the compactness of the solution operator to prove convergence of the eigenfunctions and double order convergence of eigenvalues (cf. \cite{BO}).  Additionally we also propose an a posteriori estimator of residual type, which we prove is reliable and efficient, in order  to 
perform adaptive refinements that allow to recover the optimal order of convergence for non smooth eigenfunctions. A set of representative numerical examples investigates such theoretical results.
\end{abstract}

\maketitle

\section{Introduction}\label{sec:intro}
The Stokes eigenvalue problems are important topic of research since its wide application in fluid and solid mechanics problems.
 For a given domain $\O\subset\mathbb{R}^2$ with Lipschitz boundary, the classic  Stokes eigenvalue problem (\textbf{SEP}) reads as follows:
Find $\lambda\in\mathbb{R}$, the velocity $\bu$, and the pressure $p$ such that
\begin{equation}\label{def:stokes_eigen}
\left\{
\begin{array}{rcll}
-\nu\Delta\bu +\nabla p & = & \lambda\bu&  \text{ in } \quad \Omega, \\
\div\bu & = & 0 & \text{ in } \quad \Omega, \\
\bu & = & \mathbf{0} & \text{ on } \quad \partial\Omega,\\
\ds\int_{\O}p&=&0,
\end{array}
\right.
\end{equation} 
where $\nu>0$ is the kinematic viscosity.
The bibliography related to numerical methods to solve standard Stokes eigenvalue problem \eqref{def:stokes_eigen} is very rich, 
%
(see \cite{MR3864690,MR4077220,M1402959,LEPE2023114798, MR3335223}). All these references, and the references therein, show interesting alternatives in order to obtain accurate approximations for the eigenfunctions and eigenvalues of the \textbf{SEP}. 

The virtual element method (VEM) is a numerical technique to approximate solutions of partial differential equations arising in various fields of science and engineering. This new numerical method provides many fascinating features such as a solid mathematical background, discrete schemes which are independent of particular shape of the elements, including nonconvex and oddly shaped elements, a straightforward extension to higher dimension, arbitrary order of accuracy and regularity, and convenient mesh discretization for moving boundary domain and interface problems. In view of these salient characteristic of VEM, researchers pay attention to employ this technique for different model problems in last decade. We provide here a brief list of such contribution in the VEM literature, e.g., elliptic equation \cite{MR2997471,MR3460621}, Stokes problem \cite{MR3164557,manzini2022conforming,
MR3576570}, eigenvalue problems \cite{beirao2017virtual,MR3340705}, Navier-Stokes equation \cite{da2018virtual,beirao2019stokes,adak2021virtual}. 
Moreover, recently the virtual element method has been implemented in a pseudostress formulation of the Stokes spectral problem in \cite{LEPE2021113753} where we can found evidence of the  interesting features of the VEM method in tensorial formulations, where the velocity and pressure can be recovered via  post-processing the tensorial eigenfunction. However, such a formulation is not the natural path to solve the spectral problem and therefore, is relevant to develop numerical methods to obtain directly the velocity and pressure, without the need of a postprocess technique. With this aim, we continue with our research program related to the analysis of VEM to solve the Stokes spectral problem on divergence free conforming space developed in \cite{MR3164557}. However, a direct application of the discrete space \cite{MR3164557} is not adequate since the convergence of the velocity in $\L^2$ norm is suboptimal.
This fact on the source problem is carried also to the eigenvalue problem. Hence, an important improvement of the results of \cite{MR3164557} are contained in \cite{MR4332146}, where the authors, inspired by the works of \cite{MR3796371,MR3737081}, proposed for the Stokes problem a VEM space that avoids the suboptimal order of convergence for the $\L^2$ norm of the velocity, leading to an inf-sup stable VEM that approximates the solution optimally for both, the velocity and pressure. This space, originally developed for the stationary and non-stationary Stokes source problem can be perfectly adapted for the spectral problem in which we are interested. 

Additionally, we have noticed the promising features of VEM for adaptivity and aposteriori analysis as this technique allows great flexibility of meshing without difficulty of hanging nodes, and immediate combination of different shape of elements. In fact the VEM space space \cite{MR4332146} provides a control on eigenfunctions in $\L^2$ norm which is a key feature to derive a aposteriori analysis of \textbf{SEP} (cf.~\eqref{def:stokes_eigen}). Inspired by the work \cite{MR2473688}, the last but not the least aim of this paper is to introduce and analyse a residual-type a posteriori error estimator \cite{MR3719046} for a VEM approximation of the \textbf{SEP}. To the best of authors' knowledge, this article is the first work on a posteriori analysis of \textbf{SEP} under the VEM approach.
 Therefore, developing an estimator for these discrete spaces is necessary to handle with the great flexibility of the meshes admissible by the VEM. Their adaptability becomes an attractive feature since mesh refinement strategies can be implemented in a very efficient way.
Regarding the study of the VEM to spectral problems we can cite the following manuscripts \cite{arxiv.2207.12621,MR4050542,MR4497827}.

In this article, we have developed for the first time, a combined a priori and a posteriori analysis for the \textbf{SEP} on general type of domains. We have investigated the continuous weak formulation of \eqref{def:stokes_eigen} through certain compact, continuous and self adjoint operators \cite{BO}. Based of the application of compact operator, we have defined the discrete formulation avoiding nonpolynomial part. Further the divergence free space developed in \cite{MR3164557} is not suitable for the development of a posteriori analysis since we requires the control of higher order term, i.e., the error in $\L^2$ norm. Further, by employing the regularity of the eigenfunctions, we recover the double order of convergence of the eigenvalues. Eventually, we derive a suitable residual type estimator locally and globally and prove that measuring errors with estimator is equivalent to quantify errors in the norm of the space $\H^1(\Omega) \times \L^2(\Omega)$. Based on the previous observations, we summarize our contributions to the development of divergence free VEM for the \textbf{SEP} as follows:
\begin{itemize}
\item Divergence free conforming VEM schemes are proposed for the \textbf{SEP}, and a priori, and a posteriori error estimates are derived in a unified way. 
\item The framework of the convergence analysis is robust which means the convergence analysis is derived in the norm of the continuous space.
\item We have examined several benchmark problems to justify the theory, and observed the optimal order of convergence for eigenfunctions on smooth domain. At last, we consider $L$ shaped domain with geometric singularity to verify the promising features of our proposed estimator, and recover the optimal order of convergence even in presence of er-entrant corner in the computational domain.
\end{itemize}

The outline of this manuscript is as follows. In Section \ref{sec:model_problem} we present the \textbf{SEP} in the classic saddle point formulation. We introduce the bilinear forms, stability results, the continuous solution operator, and  regularity properties for the solution.
In Section \ref{SEC:DISCRETE} the virtual element method is defined, where we introduce the properties of the mesh, local and global virtual element spaces, together with their corresponding degrees of freedom. We summarize some technical results which are needed to introduce the discrete solution operator. Also, under the framework of the compact operators theory, we prove  the convergence of the method, the absence of spurious eigenvalues, and a priori error estimates for the eigenfunctions and eigenvalues.  In Section \ref{sec:a_post} we design an a posteriori error estimator for the \textbf{SEP}, of residual type. We prove that the proposed estimator is efficient and reliable. Finally in Section \ref{sec:numerics} we report a series of numerical tests in order to assess the performance of the numerical method under analysis. These tests will illustrate the approximation of the spectrum together with the performance of the a posteriori estimator.

\subsection{Notations}
Throughout this work, $\O$ is a generic Lipschitz bounded domain of $\R^2$. For $s\geq 0$,
$\norm{\cdot}_{s,\O}$ stands indistinctly for the norm of the Hilbertian
Sobolev spaces $\HsO$ or $[\HsO]^2$ with the convention
$\H^0(\O):=\LO$.  If ${X}$ and ${Y}$ are normed vector spaces, we write ${X} \hookrightarrow {Y}$ to denote that ${X}$ is continuously embedded in ${Y}$. We denote by ${X}'$ and $\|\cdot\|_{{X}}$ the dual and the norm of ${X}$, respectively.
Finally,
we employ $\0$ to denote a generic null vector and
the relation $\texttt{a} \lesssim \texttt{b}$ indicates that $\texttt{a} \leq C \texttt{b}$, with a positive constant $C$ which is independent of $\texttt{a}$, $\texttt{b}$, and the size of the elements in the mesh. The value of $C$ might change at each occurrence. We remark that we will write the constant $C$ only when is needed.

\section{The model problem}\label{sec:model_problem}

\subsection{Weak formulation}
Let us denote by $(\cdot,\cdot)$ the classic $\L^2$- inner product. The weak spectral problem associated to \eqref{def:stokes_eigen} reads as follows: Find $\lambda\in\mathbb{R}$ and $(\boldsymbol{0},0)\neq(\bu,p)\in [\H^1_0(\O)]^2\times \L^2_0(\O)$ such that
\begin{equation}\label{eq:weak_stokes_eigen}
\begin{array}{rcll}
 a(\bu,\bv)+b(\bv,p) & = &\lambda c(\bu,\bv) &\forall\bv\in [\H_0^1(\O)]^2,\\
\ds b(\bu,q) & = & 0 &\forall\ q\in \L^2_0(\O),
\end{array}
\end{equation}
where the bilinear forms  $a:[\H^1_0(\O)]^2\times[\H^1_0(\O)]^2\rightarrow\mathbb{R}$,  $b:[\H^1_0(\O)]^2\times\L_0^2(\O)\rightarrow\mathbb{R}$ and $c:[\H^1_0(\O)]^2\times[\H^1_0(\O)]^2\rightarrow\mathbb{R}$  are defined, respectively, by
\begin{equation}
\label{eq:bilinear_forms}
\ds a(\bw,\bv):=\nu \int_{\O}\nabla\bw:\nabla\bv,
\qquad 
b(\bv,q):=-\int_{\O}q\div\bv,
\qquad  c(\bw,\bv):=\int_{\O}\bw\cdot\bv.
\end{equation}
In order to simplify the presentation of the material, we define $\bV:=[\H_0^1(\O)]^2$ and $\Q:=\L^2_0(\O)$.
%

Is a standard result that  
$a(\cdot,\cdot)$ is elliptic in $\bV$. Also, there exists $\beta>0$ such that the following inf-sup condition holds (see for instance \cite[Theorem 4.3]{MR2050138})
\begin{equation*}
\label{eq:in-sup-cont}
\ds \inf_{q\in \Q}\sup_{v\in \bV}\frac{b(\bv,q)}{\|\bv\|_{1,\O}\|q\|_{0,\O}}\geq\beta.
\end{equation*}

With these results at hand, we are in position to introduce
the solution operator $\bT$, defined as follows:
\begin{align*}
\bT:\bV\rightarrow \bV,\quad
           \boldsymbol{f}\mapsto \bT\boldsymbol{f}:=\widehat{\bu}, 
\end{align*}
\begin{equation}\label{eq:weak_stokes_source}
\begin{array}{rcll}
 a(\widehat{\bu},\bv)+b(\bv,\widehat{p}) & = &c(\boldsymbol{f},\bv) &\forall\bv\in \bV,\\
\ds b(\widehat{\bu},q) & = & 0 &\forall\ q\in \Q.
\end{array}
\end{equation}
It follows that $\bT$ is well defined due Babu\^ska-Brezzi. We observe that $(\mu,\bu)$ is an eigenpair of $\bT$ if only if there exists $(\bu,p)\in\bV\times \Q$ such that $(\lambda, (\bu,p))\in\mathbb{R}\times \bV\times \Q$ solves \eqref{eq:weak_stokes_eigen}, i.e.
\begin{equation*}
\ds \bT\bu=\mu\bu\quad\text{with}\,\,\mu:=\frac{1}{\lambda}\qquad \text{and}\,\,\l\neq0.
\end{equation*}
Moreover, it is easy to check that $\bT$ is selfadjoint with respect to the inner product in $[\L^2(\O)]^2$.

We now present an additional regularity result for problem  \eqref{eq:weak_stokes_source} and consequently for the  eigenfunctions of $\bT$, there exists  $s>0$, depending on $\O$, such that $\widehat{\bu}\in[\H^{1+s}(\O)]^2$ and $\widehat{p}\in\H^s(\O)$ and  the following stability result holds
\begin{equation}
\label{eq:data_dependence_cont}
\|\widehat{\bu}\|_{1+s,\O}+\|\widehat{p}\|_{s,\O}\lesssim \|\bF\|_{0,\O}.
\end{equation}

On the other hand, invoking the classic regularity results of \cite{MR1655512,MR606505},  there exists $r>0$,  depending on the domain, such that the eigenfunctions
$\bu\in[\H^{1+r}(\O)]^2$ and $p\in \H^r(\O)$ and satisfies
\begin{equation*}
\|\bu\|_{1+r,\O}+\|p\|_r\lesssim \|\bu\|_{0,\O}.
\end{equation*}
Therefore, due to the compact inclusion of the $[\H^{1+s}(\O)]^2\hookrightarrow [\H^1(\O)]^2$, $\bT$ is a compact operator.

Finally, we have the spectral characterization of the solution operator $\bT$.
\begin{theorem}[Spectral characterization of $\bT$]
\label{thrm:spec_char_T}
The spectrum of $\bT$ satisfies $\sp(\bT)=\{0\}\cup\{\mu_k\}_{k\in\mathbb{N}}$, where $\{\mu_k\}_{k\in\mathbb{N}}$
is a sequence of real positive eigenvalues which converges to zero, repeated according their respective multiplicities. 
\end{theorem}

Let us introduce the bilinear form $A: (\bV\times \Q)\times (\bV\times \Q)\rightarrow\mathbb{R}$  defined by
\begin{equation}
\label{eq:formaA}
A((\bu,p),(\bv,q)):= a(\bu,\bv)+b(\bv,p)+ b(\bu,q),\quad (\bu,p), (\bv,q)\in \bV\times \Q,
\end{equation}
such that problem \eqref{eq:weak_stokes_eigen} is rewritten in the following form: Find $\lambda\in\mathbb{R}$ and $(\bu,p)\in\bV\times\Q$ such that 
\begin{equation}\label{eq:eigen_A}
A((\bu,p),(\bv,q))=\lambda c(\bu,\bv)\quad\forall (\bv,q)\in\bV\times\Q.
\end{equation}
Then, from \cite{MR1463151} we have that $A(\cdot,\cdot)$ is stable in the sense that given $(\bv,q)\in\bV\times \Q $, there exists $(\bw,s)\in\bV\times \Q$ such that
$\|\bv\|_{1,\O}+\|q\|_{0,\O}\leq A((\bv,q),(\bw,s))$ and $\|\bw\|_{1,\O}+\|s\|_{0,\O}\leq C$. Also, for every $(\bv,q), (\bw,s)\in\bV\times \Q$ there holds
\begin{equation*}A((\bv,q), (\bw,s))\lesssim \|(\bv,q)\|_{ \bV\times \Q}\|(\bw,s)\|_{ \bV\times \Q}.\end{equation*}

\setcounter{equation}{0}
\section{The virtual element method}
\label{SEC:DISCRETE}

The following section is dedicated to the analysis of the virtual element approximation for the eigenproblem
presented in Problem~\ref{eq:weak_stokes_eigen}. With this aim in mind, we introduce the following definitions.
Let $\left\{\CT_h\right\}_h$ be a sequence of decompositions of $\O$
into polygons $\E$. Besides, we will denote by $\ell$ a generic edge of $\partial K$ and by $h_{\ell}$ its length. The set of all the edges in $\CT_h$ will be denote by $\mathcal{E}_h$. Let $h_\E$ denote the diameter of the element $\E$
and $h$ the maximum of the diameters of all the elements of the mesh,
i.e., $h:=\max_{\E\in\O}h_\E$. For $\CT_h$ we will consider
the following assumptions:
\begin{itemize}
\item \textbf{A1.} There exists $\g>0$ such that, for all meshes
$\CT_h$, each polygon $\E\in\CT_h$ is star-shaped with respect to a ball
of radius greater than or equal to $\g h_{\E}$.
\item \textbf{A2.} The distance between any two vertexes of $\E$ is $\geq Ch_\E$, where $C$ is a positive constant.
\end{itemize}

The bilinear forms $a(\cdot,\cdot)$, $b(\cdot,\cdot)$, $c(\cdot,\cdot)$ can be decomposed into local contributions as follows
\begin{align*}
a(\bw,\bv)&:=\sum_{\E\in\CT_{h}}a^{\E}(\bu,\bv) \qquad \text{for all }\bw, \bv\in \bV,\\
b(\bv,q)&:=\sum_{\E\in\CT_{h}}b^{\E}(\bv,q) \qquad \text{for all }\bv\in \bV \text{and }q\in \Q,\\
c(\bw,\bv)&:=\sum_{\E\in\CT_{h}}c^{\E}(\bw,\bv) \qquad \text{for all }\bw,\bv\in \bV.
\end{align*}
In the same way, we split  elementwise the norms of $[\H^1(\O)]^2$ and $\L^2(\O)$ by 
\begin{equation*}
\label{eq:normas}
\ds \|\bv\|_{1,\O}:=\left(\sum_{K\in\mathcal{T}_h}\|\bv\|_{1,K}^2\right)^{1/2}\quad\forall\bv\in \bV, \quad \|q\|_{0,\O}:=\left(\sum_{K\in\mathcal{T}_h}\|q\|_{0,K}^2\right)^{1/2}\quad\forall q\in \Q.
\end{equation*}
Further, we introduce the following function space of piecewise $\H^1$ functions
\begin{equation*}
\label{disH1Fn}
\H^1(\Omega,\CT_h):= \{ \bv \in [\L^2(\Omega)]^2~: \bv|_K \in [\H^1(K)]^2 \},
\end{equation*} associated with the semi-norm
$$|\bv|_{1,h}:=\left( \sum_{\E \in \CT_h} |\bv|_{1,\E}^2 \right)^{1/2}.$$

Inspired by  \cite{MR4332146,MR4250631}, we will construct the virtual spaces to study the velocity-pressure formulation of \eqref{eq:weak_stokes_eigen}. 
First, we introduce the velocity virtual space $\bV_{h}$: 
let  $\E$ be a simple polygon. We define 
\begin{equation*}
\label{espace-1}
\boldsymbol{\mathbb{B}}_{\partial \E}:=\{\bv_{h}\in [C^0(\partial \E)]^{2}: \bv_{h}\cdot \bn|_{\ell}\in 
\mathbb{P}_2(\ell)\text{ and }\bv_{h}\cdot \bt|_{\ell}\in 
\mathbb{P}_1(\ell)\ \ \forall \ell\subset
\partial \E\}.
\end{equation*}
With this space at hand, we consider the
following local finite dimensional 
space:
\begin{multline*}
 \label{space2}
\widehat{\bV}_h^\E:=\left\{\bv_{h}\in [\H^1(\E)]^2: \Delta \bv_{h}-\nabla s\in \nabla \mathbb{P}_2(\E)^{\perp} \text{ for some } s\in \L_0^2(E), \right.\\
\left.\div \bv_h\in \mathbb{P}_0(\E)\,\, 
\textrm{and}\,\,\bv_{h}|_{\partial \E}\in \boldsymbol{\mathbb{B}}_{\partial \E} \right\}.
 \end{multline*}
The following set of linear operators are well defined for all $\bv_h\in\widehat{\bV}_h^\E$:
 \begin{itemize}
  \item $\mathcal{V}_\E^h$ : The (vector) values of $\bv_h$ at the vertices.
  \item $\mathcal{L}_\E^h$: the value of
  \begin{equation*}
  \frac{1}{|\ell|}\int_{\ell}\bv_h\cdot\bn \quad \forall \text{ edges }\ell\in\partial\E.
  \end{equation*}.
  \item $\mathcal{\E}_\E^h$: the value of
  \begin{equation*}
  \int_\E\bv_h\cdot g^{\perp} \quad \forall \bg^{\perp}\in \nabla \mathbb{P}_2(\E)^{\perp}.
  \end{equation*}
\end{itemize}
Let us remark that the set of linear operators $\mathcal{V}_\E^h$,   $\mathcal{L}_\E^h$ and $\mathcal{\E}_\E^h$ constitutes a set of degrees of
freedom for the local virtual space $\widehat{\bV}_h^\E$ (see \cite{MR4332146}). Moreover, it is easy to check that
$[\mathbb{P}_1(\E)]^2\subset \widehat{\bV}_h^\E$. This will guarantee the good approximation properties for the space. Now we define the projector
$\boldsymbol{\Pi}^{\E}: \widehat{\bV}_h^\E\longrightarrow [\mathbb{P 
}_1(\E)]^2\subset\widehat{\bV}_h^\E$
for each $\bv_h\in\widehat{\bV}_h^\E$ as the solution of
\begin{align*}
\label{proje_0}
\left\{\begin{array}{ll}
  a^{\E}(\bp,\boldsymbol{\Pi}^{\E}\bv_h) &= a^{\E}(\bp,\bv_h)
\quad \forall \bp\in [\mathbb{P}_1(\E)]^2,\\\\
\ds\vert\partial\E\vert^{-1}\int_{\partial \E}\boldsymbol{\Pi}^{\E}\bv_h
&=\ds\vert\partial\E\vert^{-1}\int_{\partial \E}\bv_h,
\end{array}\right.
\end{align*}
We observe that, the operator $\boldsymbol{\Pi}^{\E}$ is well defined on $\widehat{\bV}_h^\E$ and computable. Now we define the local virtual element spaces $\bV_h^\E$ as follows,
\begin{equation*}
 \label{space2}
\bV_h^\E:=\left\{\bv_{h}\in \widehat{\bV}_h^\E:  \disp\int_\E \boldsymbol{\Pi}^{\E}\bv_h\cdot \bg^{\perp}=\int_\E \bv_h\cdot \bg^{\perp} \text{ for } \bg^{\perp}\in \nabla \mathbb{P}_2(\E)^{\perp}\right\}.
 \end{equation*}
Additionally, we have that the  standard $[{\mathrm L}^2(\E)]^2$-projector operator
$\boldsymbol{\Pi}_{0}^{\E}: \bV_h^{\E}\to[\mathbb{P}_1(\E)]^2$  can be  computed. In fact, for all $\bv_h\in  \bV_h^\E$, the function $\boldsymbol{\Pi}_{0}^{\E}\bv_h\in [\mathbb{P}_1(\E)]^2$ is defined by:
\begin{align*}
\label{pi0}
\int_\E\boldsymbol{\Pi}_{0}^{\E}\bv_h\cdot \boldsymbol{p}&=\int_\E\bv_h\cdot \left(\nabla p_2+\nabla p_2^{\perp}\right)\\
&=-\int_\E\div(\bv_h)p_2+\int_{\partial \E}(\bv_h\cdot\n)p_2+
 \disp\int_\E \boldsymbol{\Pi}^{\E}\bv_h\cdot\nabla p_2^{\perp}, \quad\forall \boldsymbol{p}\in[\mathbb{P}_1(\E)]^2.
\end{align*}
\begin{remark}
In the case when we need to denote the aforementioned projectors as global, we will drop the superindex $\E$ when is necessary.
\end{remark}
\noindent We can now present the global virtual space:
for every decomposition $\CT_h$ of $\O$ into
simple polygons $\E$.
\begin{align*}
\bV_h:=\left\{\bv_{h}\in\bV:\bv_h|_{\E}\in  \bV_h^{\E},\quad \forall \E\in \CT_h \right\}.
\end{align*}
In agreement with the local choice of the degrees of freedom, in 
$\bV_h$ we choose the following degrees of freedom:
\begin{itemize}
\item $\mathcal{V}^h$: the (vector) values of $\bv_h$ at the vertices of $\CT_h$.
\item $\mathcal{L}^h$: the value of
  \begin{equation*}
  \frac{1}{|\ell|} \int_\ell\bv_h\cdot\bn \quad \forall \ell\in\CT_h.
  \end{equation*}
\end{itemize}
The pressure space is given by
\begin{equation*}
\Q_h:=\{q_h \in \Q\,:\, q_h|_K\in\mathbb{P}_0(K),\,\quad\forall K\in \CT_h\},
\end{equation*}
where the degrees of freedom are one per element, given by the value of the function on the element. 


\subsection{The discrete bilinear forms}
\label{subsec:disc_b_f}
Now we will introduce the discrete version of the bilinear forms $a(\cdot, \cdot)$, $b(\cdot,\cdot)$ and $c(\cdot, \cdot)$  on $\bV_{h}^{\E}\times\bV_{h}^{\E}$ (see \cite{MR3572918}), which are symmetric and semi-positive
definite as follows. 
\begin{equation}
\label{ec:ahch}
a_h(\bw_h,\bv_h)
:=\sum_{\E\in\CT_h}a_h^{\E}(\bw_h,\bv_h),
\quad c_h(\bw_h,\bv_h)
:=\sum_{\E\in\CT_h}c_h^{\E}(\bw_h,\bv_h)\quad \bw_h,\bv_h\in\bV_{h},
\end{equation}
where $a_h^{\E}(\cdot,\cdot)$ is the bilinear form defined on
$\bV_{h}^{\E}\times\bV_{h}^{\E}$ by
\begin{equation*}
\label{21}
a_h^{\E}(\bw_{h},\bv_{h})
:=a^{\E}\big(\boldsymbol{\Pi}^{\E} \bw_h,\boldsymbol{\Pi}^{\E} \bv_h\big)
+S^{\E}\big(\bw_h-\bPi^{\E} \bw_h,\bv_h-\bPi^{\E} \bv_h\big)
\quad \bw_h,\bv_h\in\bV_{h}^{\E},
\end{equation*}
where $S^{\E}(\cdot,\cdot)$ denotes any symmetric positive definite bilinear form that satisfies
$$c_0 a^{\E}(\bw_h,\bw_h)\leq S^{\E}(\bw_h,\bw_h)\leq c_1a^{\E}(\bw_h,\bw_h),\qquad \forall \bw_h\in\bV_{h}^{\E} \cap \text{ker}(\bPi^{\E}) .$$ Further, $c_0$ and $c_1$ are positive constants depending on the mesh assumptions.
Now, we define  $c_h^{\E}(\cdot,\cdot)$ as the bilinear form defined on
$\bV_{h}^{\E}\times\bV_{h}^{\E}$ by
\begin{equation*}
\label{eq_ch}
c_h^{\E}(\bw_h,\bv_h)
:=c^{\E}\big(\bPi_{0}^{\E} \bw_h,\bPi_{0}^{\E} \bv_h\big)
\qquad \bw_h,\bv_h\in\bV_{h}^{\E}. 
\end{equation*}
%
%
%
From the definition of $c_h^\E(\cdot,\cdot)$ and the approximation properties of $\bPi_0^\E$, it is not difficult to prove that:
\begin{equation*}
c_h^\E(\boldsymbol{q}_h,\bv_h)=c^\E(\boldsymbol{q}_h,\bv_h),\qquad \boldsymbol{q}_h\in [\mathbb{P}_1(\E)]^2,
\end{equation*}
and
\begin{equation*}
\|\bv-\bPi_0^\E\bv\|_{0,\E}=\inf_{\boldsymbol{q}_h\in [\mathbb{P}_1(\E)]^2}\|\bv-\boldsymbol{q}\|_{0,\E}.
\end{equation*}
We remark that for all $\E\in\CT_h$, the local bilinear form $b^{\E}(\cdot,\cdot)$ is computable from the degrees of freedom. Since for  any function $\bv_h\in \bV_h^K$  and $q_{h}\in\Q_{h}$ we have
\begin{equation*}
b^{\E}(\bv_h, q_h)=\int_\E\div\bv_hq_h=\sum_{\ell\subset\partial \E}\int_\ell(q_{h}\bn)\cdot\bv_h.
\end{equation*}
 Moreover, $a_h^K(\cdot,\cdot)$ must admit the following properties:
\begin{itemize}
\item Consistency: For all $\boldsymbol{q}_h\in[\mathbb{P}_1(\E)]^2$ and $\bv_h\in\mathbf{V}_h^\E$
\begin{equation*}
a_h^\E(\boldsymbol{q}_h,\bv_h)=a^\E(\boldsymbol{q}_h,\bv_h),
\end{equation*}
\item Stability: There exists two positive constants $\alpha_*$ and $\alpha^*$, independent of $h$ and $\E$ such that
\begin{equation*}
\alpha_*a^\E(\bv_h,\bv_h)\leq a_h^\E(\bv_h,\bv_h)\leq\alpha^*a^\E(\bv_h,\bv_h)\quad\forall\bv_h\in\mathbf{V}_h^\E.
\end{equation*}

\end{itemize}

\begin{remark}
\label{stab:mass}
Regarding the work \cite{MR3867390}, we note that it is also possible to define the discrete bilinear form $c_h(\cdot,\cdot)$ in such a way that it remains stable, just as the bilinear form $a_h(\cdot,\cdot)$, in the following way
\begin{equation*}
c_h^{\E}(\bw_{h},\bv_{h})
:=c^{\E}\big(\boldsymbol{\Pi}_0^{\E} \bw_h,\boldsymbol{\Pi}^{\E}_0 \bv_h\big)
+S_0^{\E}\big(\bw_h-\bPi_0^{\E} \bw_h,\bv_h-\bPi_0^{\E} \bv_h\big),
\end{equation*} 
where $S_0^{\E}(\cdot,\cdot)$ denotes any symmetric positive definite bilinear form, such that there exist
two uniform positive constants $\widehat{c}_0$ and $\widehat{c}_1$ satisfying 
$$\widehat{c}_0 c^{\E}(\bw_h,\bw_h)\leq S_0^{\E}(\bw_h,\bw_h)\leq \widehat{c}_1 c^{\E}(\bw_h,\bw_h),\qquad \forall \bw_h\in\bV_{h}^{\E}\cap \ker(\bPi^{\E}_0).$$
\end{remark}
\noindent Eventually, the discrete version of \eqref{eq:weak_stokes_eigen} reads as follows: Find $\lambda_h\in\mathbb{R}$ and $(\boldsymbol{0},0)\neq (\bu_h,p_h)\in \bV_h\times \Q_h$ such that
\begin{equation}\label{eq:weak_stokes_eigen_disc}
\begin{array}{rcll}
 a_h(\bu_h,\bv_h)+b(\bv_h,p_h) & = &\lambda_h c_{h}(\boldsymbol{u}_h,\bv_h) &\forall\bv_h\in \mathbf{V}_h,\\
\ds b(\bu_h,q_h) & = & 0 &\forall\ q_h\in \Q_h.
\end{array}
\end{equation}
For the convergence analysis of the spectrum, introduce the discrete version of the operator $\bT$. To do this task, we summarize the following well known results: an inf-sup condition \cite[Lemma 4.3]{MR4332146} and a coercivity result \cite[Section 3.2]{MR4332146}. 

\begin{lemma}[discrete inf-sup]
\label{lmm:disc_inf_sup}
There exists a positive constant $\beta$, independent of $h$, such that the following inf-sup condition holds 
\begin{equation*}
\ds\sup_{\bv_h\in \bV_h, \bv_h\neq\boldsymbol{0}}\frac{b(\bv_h,q_h)}{\|\bv_h\|_{1,\O}}\geq\beta\|q_h\|_{0,\O}\qquad\forall q_h\in \Q_h.
\end{equation*}
\end{lemma}
\begin{lemma}[discrete coercivity]
\label{lm:ah_elliptic}
For $\bv_h\in\bV_h$, there exists $\tilde{\alpha}>0$, independent of $h$ such that 
\begin{equation*}
a_h(\bv_h,\bv_h)\geq \tilde{\alpha} \|\bv_h\|_{1,\O}^2\quad\forall\,\bv_h\in \bV_h.
\end{equation*}
\end{lemma}
\noindent As a consequence of Lemma \ref{lmm:disc_inf_sup} and Lemma \ref{lm:ah_elliptic}, we introduce  the discrete 
solution operator $\bT_h$ defined by 
\begin{align*}
\bT_h:\bV\rightarrow \bV_{h},\quad
           \boldsymbol{f}\mapsto \bT_h\boldsymbol{f}:=\widehat{\bu}_h, 
\end{align*}
where $\widehat{\bu}_h$ is the solution of the following discrete source problem: Given $\boldsymbol{f}\in \bV \subset [\L^2(\Omega)]^2$, find
$\widehat{\bu}_h\in \bV_h$ such that
\begin{equation}\label{eq:weak_stokes_source_disc}
\begin{array}{rcll}
 a_h(\widehat{\bu}_h,\bv_h)+b(\bv_h,\widehat{p}_h) & = &c_{h}(\boldsymbol{f},\bv_h) &\forall\bv_h\in \bV_{h},\\
\ds b(\widehat{\bu}_h,q_h) & = & 0 &\forall\ q_h\in \Q_h.
\end{array}
\end{equation}
Since the discrete  source problem above is well posed, we have the following estimate  (see \cite[Theorem 4.4]{MR3626409})
\begin{equation*}
\|\widehat{\bu}_h\|_{1,\O}+\|\widehat{p}_h\|_{0,\O}\lesssim \|\boldsymbol{f}\|_{0,\O}.
\end{equation*}
Further from the construction of discrete space, we note that $\div\bV_h=\Q_h$ (see \cite[equation (4.17)]{MR3626409}).
As in the continuous case,  it is easy to check that $(\mu_h,\bu_h)$ is an eigenpair of $\bT_h$ if only if there exists $(\bu_h,p_h)\in\bV_h\times \Q_h$ such that $(\lambda_h, (\bu_h,p_h))\in\mathbb{R}\times \bV_h\times \Q_h$ solves \eqref{eq:weak_stokes_eigen_disc}, i.e.
\begin{equation*}
\ds \bT_h\bu_h=\mu_h\bu_h\quad\text{with}\,\,\mu_h:=1/\lambda_h\qquad \text{and}\,\,\l_h\neq0.
\end{equation*}
For better presentation of the article, we rewrite the problem \eqref{eq:weak_stokes_source_disc} as follows: Given $\boldsymbol{f}\in[\L^2(\O)]^2$, find  $(\bu_h,p_h)\in \bV_{h}\times \Q_h$ such that
\begin{equation*}
\label{eq:formulation_A_h}
A_h((\bu_h,p_h),(\bv_h,q_h))=c_{h}(\boldsymbol{f},\bv_h),\quad (\bv_h,q_h)\in \bV_{h}\times \Q_h,
\end{equation*}
where $A_h: (\bV_{h}\times \Q_h)\times (\bV_{h}\times \Q_h)\rightarrow\mathbb{R}$ is defined by
 \begin{equation}
 \label{ec:forma_Ah}
A_h((\bu_h,p_h),(\bv_h,q_h)):= a_h(\bu_h,\bv_h)+b(\bv_h,p_h)+ b(\bu_h,q_h),
\end{equation}
for all  $(\bu_h,p_h), (\bv_h,q_h)\in \bV_{h}\times \Q_h.$ Then, from the proof of \cite[Theorem 3.1]{MR3164557} we have that $A_h(\cdot,\cdot)$ is stable in the sense that given $(\bv_h,q_h)\in\bV_h\times \Q_h$, there exists $(\bw_h,s_h)\in\bV_h\times \Q_h$ such that
$\|\bv_h\|_{1,\O}+\|q_h\|_{0,\O}\leq A_h((\bv_h,q_h),(\bw_h,s_h))$ and $\|\bw_h\|_{1,\O}+\|s_h\|_{0,\O}\leq C$. Also, for every $(\bv_h,q_h), (\bw_h,s_h)\in\bV_h\times \Q_h$ there holds
\begin{equation*}A_h((\bv_h,q_h), (\bw_h,s_h))\lesssim \|(\bv_h,q_h)\|_{\bV\times \Q}\|(\bw_h,s_h)\|_{ \bV\times \Q}.\end{equation*}

\subsection{Technical approximation results}
The following approximation result for
polynomials in star-shaped domains (see for instance \cite{MR2373954}),
 is derived from results of interpolation between Sobolev spaces
(see for instance \cite[Theorem~I.1.4]{MR851383}) leading to an analogous
result for integer values of $s$. 

\begin{lemma}[Existence of a virtual approximation operator]
\label{lmm:bh}
If assumption {\bf A1} is satisfied, then  for every $s$ with 
$0\le s\le 1$ and for every $\bv\in[\H^{1+s}(K)]^2$, there exists
$\bv_{\pi}\in\mathbb{P}_1(\E)$ such that
$$
\left\|\bv-\bv_{\pi}\right\|_{0,\E}
+h_{\E}\left|\bv-\bv_{\pi}\right|_{1,\E}
\lesssim  h_{\E}^{1+s}\left\|\bv\right\|_{1+s,\E},
$$
where the hidden constant  depends only on mesh regularity constant  $\gamma$.
\end{lemma}


\noindent We also have the following result that provides the existence of an interpolated function for the velocity, together with an approximation error (see \cite[Lemma  4.2]{MR4332146}).
\begin{lemma}[Existence of an interpolation operator]
\label{estimate2}
Under the assumptions  {\bf A1} and  {\bf A2}, let $\bv\in \bV\cap[\H^{1+s}(\O)]^2$, with $0\leq s\leq 1$. Then, there exists $\bv_I\in \bV_h$ such that
\begin{equation*}
\|\bv-\bv_I\|_{0,K}+h_K|\bv-\bv_I|_{1,K}\lesssim h_K^{1+s}|\bv|_{1+s,K},
\end{equation*}
where the hidden constant is  positive and independent of $h_K$.
\end{lemma}
\noindent Finally, let $\mathcal{P}_{h}:\L^2(\O)\to \left\{q\in \L^2(\O): q|_{\E}\in\mathbb{P}_{0}(\E)\right\}$  be the $\L^2(\O)$-orthogonal projector which satisfies the following approximation property.
\begin{lemma}
\label{estimate3}
If $0\leq s\leq 1$, it holds
\begin{equation*}
\|q-\mathcal{P}_{h}(q)\|_{0,\O}\lesssim h^{s}\|q\|_{s,\O} \qquad \forall q\in\H^s(\O)\cap \Q.
\end{equation*}
\end{lemma}
\subsection{Spectral approximation}
In what follows we  will focus in the approximation between $\bT$ and $\bT_h$, implying the analysis of convergence for the eigenvalues and the respective eigenfunctions. Since $\bT$ is a compact operator, it is sufficient to analyze the spectral correctness with the classic theory of \cite{BO}, implying as a first step, to obtain the convergence in norm between $\bT_h$ and $\bT$, which immediately
yields to the convergence of the eigenvalues and eigenfunctions, with no spurious eigenvalues.
\begin{remark}
In the discretization of $c^{\E}(\cdot,\cdot)$, we have considered the polynomial part of the discrete bilinear form. With this discretization, it is justified to define the solution operator $\bT$ on $[\H^1(\Omega)]^2$, and to employ the spectral theory for compact operator \cite{BO}. However, one can discretize $c^{\E}(\cdot,\cdot)$ following Remark~\ref{stab:mass}. In such case, we can derive the convergence analysis by following the spectral theory for nonselfadjoint operator \cite{MR447842}.
\end{remark}
\noindent We begin by proving the convergence in norm between the continuous and discrete solution operators.
\begin{lemma}[Convergence of $\bT_h$ to $\bT$]
\label{lmm:converg1}
There exists $s>0$ such that
\begin{equation*}
\|(\bT-\bT_h)\bF\|_{1,\O}=\|\widehat{\bu}-\widehat{\bu}_h\|_{1,\O}\leq \|\widehat{\bu}-\widehat{\bu}_h\|_{1,\O}+\|\widehat{p}-\widehat{p}_h\|_{0,\O}\lesssim  h^{\min\{1,s\}}\|\bF\|_{1,\O},
\end{equation*}
where the hidden constant is independent of $h$.
\end{lemma}
\begin{proof}
Let $\boldsymbol{f}\in[\H^1(\O)]^2$ be such that $\widehat{\bu}:=\bT\boldsymbol{f}$ and $\widehat{\bu}_h:=\bT\boldsymbol{f}$, where $(\widehat{\bu},\widehat{p})\in[\H_0^1(\O)]^2\times\L^2_0(\O)$
is the solution of \eqref{eq:weak_stokes_source} and $(\widehat{\bu}_h,\widehat{p}_h)\in\bV\times\Q_h$ is the solution of \eqref{eq:weak_stokes_source_disc}. Then, we have
\begin{multline*}
\|(\bT-\bT_h)\bF\|_{1,\O}\leq \|\widehat{\bu}-\widehat{\bu}_h\|_{1,\O}+\|\widehat{p}-\widehat{p}_h\|_{0,\O}\\
\leq \|\widehat{\bu}-\widehat{\bu}_I\|_{1,\O}+\|\widehat{p}-\Ph(\widehat{p})\|_{0,\O}
+\underbrace{\|\widehat{\bu}_I-\widehat{\bu}_h\|_{1,\O}+\|\Ph(\widehat{p})-\widehat{p}_h\|_{0,\O}}_{\mathbf{(I)}}.
\end{multline*}
The aim now is to estimate each of the contributions of the right hand side  on the inequality above. We observe that invoking Lemma \ref{estimate2} with $s=0$ and \eqref{eq:data_dependence_cont}, we obtain the estimate  
\begin{equation}\label{eq:first_eq}
\|\widehat{\bu}-\widehat{\bu}_I\|_{1,\O}\lesssim h^s\|\widehat{\bu}\|_{1+s,\O}\lesssim h^s\|\bF\|_{0,\O}.
\end{equation}
On the other hand, using Lemma \ref{estimate3} we have
\begin{equation}\label{eq:second_eq}
\|\widehat{p}-\Ph(\widehat{p})\|_{0,\O}\lesssim h^s\|\widehat{p}\|_{s,\O}\lesssim h^s\|\bF\|_{0,\O}.
\end{equation}
Now our aim is to estimate $\mathbf{(I)}$. To do this  task, we resort to the following  estimate, consequence of the stability of the discrete problem (see the proof of \cite[Theorem 3.1]{MR3164557}
$$\|\widehat{\bu}_I-\widehat{\bu}_h\|_{1,\O}+\|\Ph(\widehat{p})-\widehat{p}_h\|_{0,\O}\lesssim  A_h((\widehat{\bu}_I-\widehat{\bu}_h,\Ph(\widehat{p})-\widehat{p}_h),(\bv_h,q_h)),$$
with $\|\bv_h\|_{1,\O}+\|q_h\|_{0,\O}\leq 1$. Hence,
 
\begin{multline*}
\mathbf{(I)}\lesssim  A_h((\widehat{\bu}_I,\Ph(\widehat{p})-\widehat{p}_h),(\bv_h,q_h))-c_h(\bF,\bv_h)\\
=\sum_{\E\in\CT_h}[a_h^\E(\widehat{\bu}_I-\widehat{\bu}_\pi,\bv_h)+a^E(\widehat{\bu}_\pi-\widehat{\bu},\bv_h)+a^\E(\widehat{\bu},\bv_h)
+b^{\E}(\bv_h,\widehat{p}_h)+b^{\E}(\widehat{\bu}_I.q_h)]-c_h(\bF,\bv_h)\\
=\sum_{\E\in\CT_h}c^{\E}(\bF,\bv_h-\Pi_0^\E\bv_h)+\sum_{\E\in\CT_h}[ a_h^{\E}(\widehat{\bu}_I-\widehat{\bu}_{\pi},\bv_h)+a^{\E}(\widehat{\bu}_{\pi}-\widehat{\bu},\bv_h)\\
+b^{\E}(\bv_h,\Ph(\widehat{p})-\widehat{p})
+b^{\E}(\widehat{\bu}_I-\widehat{\bu},q_h)]\\
\lesssim \sum_{\E\in\CT_h}\|\bF\|_{0,\E}\|\bv_h-\bPi_0^\E\bv_h\|_{0,\E}+\sum_{\E\in\CT_h}(|\widehat{\bu}_I-\widehat{\bu}|_{1,\E}+|\widehat{\bu}-\widehat{\bu}_{\pi}|_{1,\E})|\bv_h|_{1,\E}\\
+\sum_{\E\in\CT_h}|\widehat{\bu}_\pi-\widehat{\bu}|_{1,\E}|\bv_h|_{1,\E}+\sum_{\E\in\CT_h}\|\Ph(\widehat{p})-\widehat{p}\|_{0,\E}|\bv_h|_{1,\E}+\sum_{\E\in\CT_h}|\widehat{\bu}_I-\widehat{\bu}|_{1,\E}\|q_h\|_{0,\E}\\
\lesssim  \sum_{\E\in\CT_h}\left(h_{\E}\|\bF\|_{0,\E}+h_{\E}^s\|\widehat{\bu}\|_{1+s,\E}+h_{\E}^s\|\widehat{p}\|_{s,\E}\right)\|\bv_h\|_{1,\E}+\sum_{\E\in\CT_h}h_{\E}^s\|\widehat{\bu}\|_{1+s,\E}\|q_h\|_{0,\E}\\
\lesssim h^{\min\{1,s\}}\|\bF\|_{0,\O},
\end{multline*}
where the last estimate is a consequence of $\|\bv_h\|_{1,\O}+\|q_h\|_{0,\O}\leq 1$, Lemmas \ref{lmm:bh}--\ref{estimate3}, and  \eqref{eq:data_dependence_cont}. Hence, the proof is derived from the above estimate together with \eqref{eq:first_eq} and \eqref{eq:second_eq}.
\end{proof}

\subsection{Error estimates}
\label{SEC:BUCKL:SPEC}
As a direct consequence of Lemma~\ref{lmm:converg1}, standard results about
spectral approximation (see \cite{MR1335452}, for instance) show that isolated
parts of $\sp(\bT)$ are approximated by isolated parts of $\sp(\bT_h)$. More
precisely, let $\mu\in(0,1)$ be an isolated eigenvalue of $\bT$ with
multiplicity $m$ and let $\mathfrak{E}$ be its associated eigenspace. Then, there
exist $m$ eigenvalues $\mu^{(1)}_h,\dots,\mu^{(m)}_h$ of $\bT_h$ (repeated
according to their respective multiplicities) which converge to $\mu$.
Let $\mathfrak{E}_h$ be the direct sum of their corresponding associated
eigenspaces. Further, we recall the definition of the \textit{gap} $\hdel$ between two closed
subspaces $\CM$ and $\CN$ of $\H_0^1(\O)$:
$$
\hdel(\CM,\CN)
:=\max\left\{\delta(\CM,\CN),\delta(\CN,\CM)\right\},$$
where
$$\delta(\CM,\CN)
:=\sup_{x\in\CM:\ \left\|x\right\|_{1,\O}=1}
\left(\inf_{y\in\CN}\left\|x-y\right\|_{1,\O}\right).
$$
With these definitions at hand, now we are in a position to establish error estimates for the approximation of the eigenspaces. 
\begin{theorem}[Error estimates]
\label{thm:error_spaces}
The following estimates hold
\begin{align*}
\widehat{\delta}(\mathfrak{E},\mathfrak{E}_h)\lesssim h^{\min\{1,r\}}\quad\text{and}\quad
|\mu-\mu_h^{(i)}|\lesssim h^{\min\{1,r\}},
\end{align*}
where the hidden constants are independent of $h$.
\end{theorem}
\begin{proof}
The proof is a direct consequence of Lemma \ref{lmm:converg1} and \cite{BO}. 
\end{proof}
\noindent Also, the classic double order of convergence for the eigenvalues is derived and proved in the following result.
\begin{theorem}(Double order of convergence for eigenvalues)
\label{cotadoblepandeo} 
For all $r>0$, there holds
$$\
\left|\l-\l_h^{(i)}\right|
\lesssim h^{2\min\{1,r\}},
$$
where the hidden constant is strictly positive and independent of $h$.
\end{theorem}
\begin{proof}
In order to simplify the presentation of the material, let us define 
$$\bx:=(\bu,p),\quad \by:=(\bv,q)\in\bV\times \Q, $$
and  the corresponding  discrete counterparts
$$\bx_h:=(\bu_h,p_h),\quad \by_h:=(\bv_h,q_h)\in\bV_{h}\times \Q_h,$$ 
where $\bx$ and $\bx_h$ are the solutions of the following spectral problems
\begin{equation*}
A(\bx,\by)=\lambda c(\bu,\bv)\,\,\,\,\,\,\forall\by\in\bV\times \Q,
\end{equation*}
and
\begin{equation*}
A_h(\bx_h,\by_h)=\lambda_h^{(i)} c_h(\bu_h,\bv_h)\,\,\,\,\,\,\,\forall\by_h\in\bV_{h}\times \Q_h, 
\end{equation*}
where $A(\cdot,\cdot), c(\cdot,\cdot), A_h(\cdot,\cdot)$ and $c_h(\cdot,\cdot)$ are as in \eqref{eq:formaA}, \eqref{eq:bilinear_forms}, \eqref{ec:forma_Ah} and 
\eqref{ec:ahch}, respectively. Let $\bu_h\in\mathfrak{E}_h$ be an eigenfunction corresponding to the eigenvalue $\lambda_h^{(i)}$ with $i=1,\ldots, m$ and $\|\bu_h\|_{1,\O} =1$. Since $\widehat{\delta}\left(\bu_h,\mathfrak{E}\right)\lesssim h^r$, there exists $\bu\in\mathfrak{E}$ such that
\begin{equation*}
\|\bu-\bu_h\|_{1,\O} +\|p-p_h\|_{0,\O}\lesssim h^{\min\{1,r\}}.
\end{equation*}
Since the bilinear forms are symmetric, the following identity holds
\begin{multline*}
|\lambda-\lambda_h| c(\bu_h,\bu_h)=\underbrace{\sum_{\E\in\CT_h} a_h^{\E}(\bu_h,\bu_h)-a^{\E}(\bu_h,\bu_h)}_{\textbf{I}}
+\underbrace{\lambda_h[c(\bu_h,\bu_h)-c_h(\bu_h,\bu_h) ]}_{\textbf{II}} \\
+\underbrace{A((\bu-\bu_h,p-p_h),(\bu-\bu_h,p-p_h))}_{\textbf{III}}-\underbrace{\lambda c(\bu-\bu_h,\bu-\bu_h)}_{\textbf{IV}}.
\end{multline*}
An application of consistency, stability, and polynomial approximation properties of $a_h^K(\cdot,\cdot)$ yield the control on each terms on the identity above. We observe that for term $\textbf{I}$, from the consistency,  we have
\begin{multline}
\label{eq:term_I}
|\textbf{I}|=\left|\sum_{\E\in\CT_h}(a_h^\E(\bu_h-\bu_\pi,\bu_h)-a^\E(\bu_h-\bu_\pi,\bu_h)\right|\\
=\left|\sum_{\E\in\CT_h}(a_h^\E(\bu_h-\bu_\pi,\bu_h-\bu_\pi)-a^\E(\bu_h-\bu_\pi,\bu_h-\bu_\pi))\right|\\
\lesssim \sum_{\E\in\CT_h}|\bu_h-\bu|_{1,\E}^2+\sum_{\E\in\CT_h}|\bu-\bu_\pi|_{1,\E}^2\lesssim h^{2\min\{1,r\}}.
\end{multline}
By exploiting the orthogonality property of $\bPi_0^\E$, we derive the bound for the term $\mathbf{II}$ as follows
\begin{multline}
\label{eq:term_II}
|\mathbf{II}|=\left|\int_{\E}|\bu_h|^2-\int_{\E}\bPi_0^\E\bu_h\cdot\bu_h \right|=\left|\int_\E(\bu_h-\bPi_0^\E\bu_h)\cdot(\bu_h-\bPi_0^\E\bu_h)\right|\\
\lesssim \|\bu_h-\bPi_0^\E\bu_h\|_{0,\E}^2\lesssim h^{2\min\{1,r\}}.
\end{multline}
On the other hand, from Theorem \ref{thm:error_spaces} it is asserted that 
\begin{equation}
\label{eq:term_III}
|\mathbf{III}|\lesssim \|\bu-\bu_h\|_{1,\O}^2+\|p-p_h\|_{0,\O}^2\lesssim h^{2\min\{1,r\}},
\end{equation}
whereas for $\mathbf{IV}$, we have
\begin{equation}
\label{eq:term_IV}
|\mathbf{IV}|\lesssim \|\bu-\bu_h\|_{0,\O}^2\lesssim h^{2\min\{1,r\}}.
\end{equation}
Finally we deduce that from the definition of $A_h(\cdot,\cdot)$, its stability property (cf. subsection \ref{subsec:disc_b_f}), the fact that $\|\bu_h\|_{1,\O}=1$,   and the convergence of eigenvalues   $\lambda_h^{(i)}\rightarrow \lambda$ as $h\rightarrow 0$, we have 
\begin{equation}
\label{eq:termV}
c(\bu_h,\bu_h)\geq\frac{A_h((\bu_h,p_h), (\bu_h,p_h))}{\lambda_h^{(i)}} \geq\frac{\|\bu_h\|_{1,\O}+\|p_h\|_{0,\O}}{\lambda_h^{(i)}}=\frac{1+\|p_h\|_{0,\O}}{\lambda_h^{(i)}}\geq \widetilde{C}>0.
\end{equation}
Hence, gathering \eqref{eq:term_I}, \eqref{eq:term_II}, \eqref{eq:term_III}, \eqref{eq:term_IV}, and \eqref{eq:termV}, we conclude the proof.
\end{proof}

\begin{remark}\label{reamrk:ord2}
From the above proof it is easy to note that
\begin{equation*}
\left|\l-\l_h^{(i)}\right| \lesssim \|\bu-\bu_h\|_{1,\O}^2+\|p-p_h\|_{0,\O}^2+\sum_{\E\in\CT_h}\left(|\bu-\bPi^\E\bu_h|_{1,\E}^2+\|\bu-\bPi_0^\E\bu_h\|_{0,\E}^2\right).
\end{equation*}
This estimate will be useful for the a posteriori error analysis.
\end{remark}
\noindent Now we present an error estimate for the velocity in $\L^2$ norm. The proof for this result is based on a classic duality argument.
\begin{lemma}\label{lmm:dualidad_t}
Under assumptions $\textbf{A1}$ and $\textbf{A2}$, given  $\bF\in\mathfrak{E}$ such that $\widehat{\bu}:=\bT\bF$ and $\widehat{\bu}_h:=\bT_h\bF$, there holds
\begin{equation*}
\|\widehat{\bu}-\widehat{\bu}_h\|_{0,\O}\lesssim h^s\left(\|\widehat{\bu}-\widehat{\bu}_h\|_{1,\O}+\|\widehat{\bu}-\bPi_0\widehat{\bu}\|_{0,\O}+|\widehat{\bu}-\bPi\widehat{\bu}_h|_{1,h}+\|\widehat{p}-\widehat{p}_h\|_{0,\O}\right),
\end{equation*}
where the hidden constant is independent of $h$.
\end{lemma}
\begin{proof}
Let us consider the following well posed problem: find $(\boldsymbol{z},\phi)\in\bV\times\Q$ such that 
\begin{equation}\label{eq:weak_stokes_dual}
\begin{array}{rcll}
 a(\boldsymbol{\sigma},\boldsymbol{z})+b(\boldsymbol{\sigma},\phi) & = &c(\widehat{\bu}-\widehat{\bu}_h,\boldsymbol{\sigma}) &\forall\boldsymbol{\sigma}\in \bV,\\
\ds b(\boldsymbol{z},r) & = & 0 &\forall\ r\in \Q,
\end{array}
\end{equation}
where the solution of \eqref{eq:weak_stokes_dual} satisfies
\begin{equation}
\label{eq:data_dependence_cont_dual}
\|\boldsymbol{z}\|_{1+s,\O}+\|\phi\|_{s,\O}\lesssim \|\widehat{\bu}-\widehat{\bu}_h\|_{0,\O}.
\end{equation}
Now our task is to estimate the right hand side of the inequality above. Elementary computations reveal
\begin{multline}\label{eq:dualidad1}
\|\widehat{\bu}-\widehat{\bu}_h\|_{0,\O}^2= c(\widehat{\bu}-\widehat{\bu}_h,\widehat{\bu}-\widehat{\bu}_h)=a(\widehat{\bu}-\widehat{\bu}_h,\boldsymbol{z})+b(\widehat{\bu}-\widehat{\bu}_h,\phi)\\
=a(\widehat{\bu}-\widehat{\bu}_h,\boldsymbol{z}-\boldsymbol{z}_I)+a(\widehat{\bu}-\widehat{\bu}_h,\boldsymbol{z}_I)+b(\widehat{\bu}-\widehat{\bu}_h,\phi-\Ph(\phi))+b(\widehat{\bu}-\widehat{\bu}_h,\Ph(\phi)).
\end{multline}
Regarding the interpolation properties of $\boldsymbol{z}_I$, Cauchy-Schwarz inequality, and \eqref{eq:data_dependence_cont_dual}, we obtain the following estimates
\begin{align*}
a(\widehat{\bu}-\widehat{\bu}_h,\boldsymbol{z}-\boldsymbol{z}_I)\lesssim h^s|\widehat{\bu}-\widehat{\bu}_h|_{1,\O}\|\widehat{\bu}-\widehat{\bu}_h\|_{0,\O},\\
b(\widehat{\bu}-\widehat{\bu}_h,\phi-\Ph(\phi))\lesssim h^s|\widehat{\bu}-\widehat{\bu}_h|_{1,\O}\|\widehat{\bu}-\widehat{\bu}_h\|_{0,\O}.
\end{align*}
Moreover, the identity $b(\widehat{\bu}-\widehat{\bu}_h,\Ph(\phi))=0$,  is a consequence of the second equation of \eqref{eq:weak_stokes_source} and \eqref{eq:weak_stokes_source_disc}. 
On the other hand, 
\begin{equation*}
a(\widehat{\bu}-\widehat{\bu}_h,\boldsymbol{z}_I)=\underbrace{c(\bF,\boldsymbol{z}_I)-c_h(\bF,\boldsymbol{z}_I)}_{\Lambda_1}+\underbrace{a_h(\widehat{\bu}_h,\boldsymbol{z}_I)-a(\widehat{\bu},\boldsymbol{z}_I)}_{\Lambda_2}+\underbrace{b(\boldsymbol{z}_I-\boldsymbol{z},\widehat{p}_h-\widehat{p})}_{\Lambda_3}.
\end{equation*}
Then, for $\Lambda_1$, from the stability of $\bPi_0^\E$, the error  interpolation properties of $\boldsymbol{z}_I$ and \eqref{eq:data_dependence_cont_dual}, we have
\begin{multline*}
\Lambda_1=\sum_{\E\in\CT_h}c^\E(\bF,\boldsymbol{z}_I)-c_h^\E(\bF,\boldsymbol{z}_I)=\sum_{\E\in\CT_h}c^\E(\bF-\bPi_0^\E\bF,\boldsymbol{z}_I-\bPi_0^\E\boldsymbol{z}_I)\\
\leq\sum_{\E\in\CT_h}\|\bF-\bPi_0^\E\bF\|_{0,\E}\|\boldsymbol{z}_I-\bPi_0^\E\boldsymbol{z}_I\|_{0,\E}\\
\leq\sum_{\E\in\CT_h}\|\bF-\bPi_0^\E\bF\|_{0,\E}(\|\boldsymbol{z}_I-\boldsymbol{z}\|_{0,\E}+\|\boldsymbol{z}-\bPi_0^\E\boldsymbol{z}\|_{0,\E}+\|\bPi_0^\E(\boldsymbol{z}-\boldsymbol{z}_I)\|_{0,\E})\\
\lesssim h^s\left(\sum_{\E\in\CT_h}\|\bF-\bPi_0^\E\bF\|_{0,\E}^2\right)^{1/2}\|\widehat{\bu}-\widehat{\bu}_h\|_{0,\O}.
\end{multline*}
Now, use the fact that $\bF\in \mathfrak{E}$, $\widehat{\bu}=\bT\bF=\mu\bF$ to obtain
\begin{multline}
\label{eq:dualidad2}
\Lambda_1\lesssim h^s\left(|\mu^{-1}|\sum_{\E\in\CT_h}\|\widehat{\bu}-\bPi_0^\E\widehat{\bu}\|_{0,\E}^2\right)^{1/2}\|\widehat{\bu}-\widehat{\bu}_h\|_{0,\O}
\lesssim h^r\|\widehat{\bu}-\bPi_0^\E\widehat{\bu}\|_{0,\O}\|\widehat{\bu}-\widehat{\bu}_h\|_{0,\O}.
\end{multline}
For $\Lambda_2$, there holds
\begin{multline}\label{eq:dualidad3}
\Lambda_2=\sum_{\E\in\CT_h}a_h^\E(\widehat{\bu}_h,\boldsymbol{z}_I)-a^\E(\widehat{\bu},\boldsymbol{z}_I)\\
=\sum_{\E\in\CT_h}
a_h^\E(\widehat{\bu}_h-\bPi^\E\widehat{\bu}_h,\boldsymbol{z}_I)+a_h(\bPi^\E\widehat{\bu}_h,\boldsymbol{z}_I)-a^\E(\widehat{\bu}-\bPi^\E\widehat{\bu}_h,\boldsymbol{z}_I)-a^\E(\bPi^\E\widehat{\bu}_h,\boldsymbol{z}_I)\\
=\sum_{\E\in\CT_h} a_h^\E(\widehat{\bu}_h-\bPi^\E\widehat{\bu}_h,\boldsymbol{z}_I-\bPi^\E\boldsymbol{z}_I)-a^\E(\widehat{\bu}_h-\bPi^\E\widehat{\bu}_h,\boldsymbol{z}_I-\bPi\boldsymbol{z}_I)\\
\lesssim\sum_{\E\in\CT_h}|\widehat{\bu}_h-\bPi^\E\widehat{\bu}_h|_{1,\E}|\boldsymbol{z}_I-\bPi^\E\boldsymbol{z}_I|_{1,\E}\lesssim\sum_{\E\in\CT_h}|\widehat{\bu}_h-\bPi^\E\widehat{\bu}_h|_{1,\E}h_\E^s|z|_{1,\E}\\
\lesssim h^s(|\widehat{\bu}_h-\widehat{\bu}|_{1,\O}+|\widehat{\bu}-\bPi\widehat{\bu}_h|_{1,h})\|\widehat{\bu}-\widehat{\bu}_h\|_{0,\O}.
\end{multline}
Now, $\Lambda_3$ is easily estimated by 
\begin{equation*}
\Lambda_3\lesssim \|\boldsymbol{z}_I-\boldsymbol{z}\|_{1,\O}\|\widehat{p}-\widehat{p}_h\|_{0,\O}\lesssim h^s\|\boldsymbol{z}\|_{1+s,\O}\|\widehat{p}-\widehat{p}_h\|_{0,\O}
\lesssim h^s\|\widehat{p}-\widehat{p}_h\|_{0,\O}\|\widehat{\bu}-\widehat{\bu}_h\|_{0,\O}.
\end{equation*}
Finally, combining the above estimate with \eqref{eq:dualidad1}-\eqref{eq:dualidad3} allows us to conclude the proof.
\end{proof}

\noindent We are now in a position to define a solution operator on the space $[\L^2(\O)]^2$:
$$\widetilde{\bT}:[\L^2(\O)]^2\to [\L^2(\O)]^2,\qquad \widetilde{\bF}\longmapsto \widetilde{\bT}\widetilde{\bF}:=\widetilde{\bu},$$
where the pair $(\widetilde{\bu},\widetilde{p})$ is the solution of \eqref{eq:weak_stokes_source}. It is easy to check that the operator $\widetilde{\bT}$ is self-adjoint and compact. Moreover, the spectra of $\widetilde{\bT}$ and $\bT$ coincide.
\begin{remark}\label{rm:L2}
Since the spectra of $\widetilde{\bT}$ and $\bT$ coincide, as a consequence of Lemma  \ref{lmm:converg1}, we have that for all $\bF\in[\L^2(\O)]^2$, the following estimate hold
\begin{equation*}
\|(\widetilde{\bT}-\bT_h)\bF\|_{0,\O}\lesssim h^s\|\bF\|_{0,\O},
\end{equation*}
where the hidden constant is independent of $h$.
\end{remark}

\begin{lemma}
Let $\bu_h$ be an eigenfunction of $\bT_h$ associated with eigenvalues $\mu_h^i$, $1\leq i\leq m$, with $\|\bu_h\|_{0,\O}=1$. Then there exists  an eigenfunction $\bu\in[\L^2(\O)]^2$ of $\bT$ associated with $\mu$ such that
\begin{equation*}
\|\bu-\bu_h\|_{0,\O}\lesssim h^s\left(\|\widehat{\bu}-\widehat{\bu}_h\|_{1,\O}+\|\widehat{\bu}-\bPi_0\widehat{\bu}\|_{0,\O}+|\widehat{\bu}-\bPi\widehat{\bu}_h|_{1,h}+\|\widehat{p}-\widehat{p}_h\|_{0,\O}\right).
\end{equation*}
\end{lemma}
\begin{proof} The spectral convergence of $\bT_h$ at $\widetilde{\bT}$, is a consequence of \cite{BO} and  Remark \ref{rm:L2}. In particular, because of the relation between the eigenfunctions of $\bT$ and $\bT_h$ with those of $\widetilde{\bT}$ and $\bT_h$, respectively, we have $\bu_h\in \mathfrak{E}_h$ and there exists $\bu\in \mathfrak{E}$ such that
\begin{equation*}
\|\bu-\bu_h\|_{0,\O}\lesssim \displaystyle\sup_{\widetilde{\bF}\in \mathfrak{E}_h:\|\widetilde{\bF}\|_{0,\O}=1}\|(\widetilde{\bT}-\bT_h)\widetilde{\bF}\|_{0,\O}.
\end{equation*}
On the other hand, from Lemma \ref{lmm:dualidad_t}, if $\widetilde{\bF}\in\mathfrak{E}_h$ is such that $\widetilde{\bF}=\bF$, then
\begin{multline*}
\|(\widetilde{\bT}-\bT_h)\widetilde{\bF}\|_{0,\O}=\|(\bT-\bT_h)\bF\|_{0,\O}		
\\
\lesssim h^s\left(\|\widehat{\bu}-\widehat{\bu}_h\|_{1,\O}+\|\widehat{\bu}-\bPi_0\widehat{\bu}\|_{0,\O}+|\widehat{\bu}-\bPi\widehat{\bu}_h|_{1,h}+\|\widehat{p}-\widehat{p}_h\|_{0,\O}\right),
\end{multline*}
which conclude the proof.
\end{proof}
\noindent Now we are in a position to prove the following result, which will be crucial in order to derive the reliability of the residual estimator. 
\begin{theorem}\label{The:L2up}
There exists $s>0$ such that
\begin{equation*}
\|\bu-\bu_h\|_{0,\O}\lesssim h^s\left(\|\bu-\bu_h\|_{1,\O}+\|p-p_h\|_{0,\O}+\|\bu-\bPi_0\bu_h\|_{0,\O}+|\bu-\bPi\bu_h|_{1,h}\right).
\end{equation*}
\end{theorem}
\begin{proof}

Let us notice that
\begin{equation}\label{eq:aux0}
\|\bu-\bu_h\|_{0,\O}\lesssim h^s\left(\|\widehat{\bu}-\widehat{\bu}_h\|_{1,\O}+\|\widehat{p}-\widehat{p}_h\|_{0,\O}+\|\widehat{\bu}-\bPi_0\widehat{\bu}\|_{0,\O}+|\widehat{\bu}-\bPi\widehat{\bu}_h|_{1,h}\right).
\end{equation}
Our aim is to control each term that appears on the right-hand side of the above estimate. Let  $(\widehat{\bu},\widehat{p})\in\bV\times\Q$ be the solution of the following problem:
\begin{equation*}
\label{eq:aux2}
A((\widehat{\bu},\widehat{p}),(\bv,q))=c(\bu,\bv)\quad\forall (\bv,q)\in\bV\times\Q.
\end{equation*}
and  recalling that  $(\bu,p)\in\bV\times\Q$ satisfies
\begin{equation*}
\label{eq:aux1}
A((\bu,p),(\bv,q))=\lambda c(\bu,\bv)\quad\forall (\bv,q)\in\bV\times\Q,
\end{equation*}
Then, from these two problems, it is clear that $\widehat{\bu}=\bu/\lambda$ and $\widehat{p}=p/\lambda$. Now, let  $(\widehat{\bu}_h,\widehat{p}_h)\in\bV_h\times\Q_h$ be the  unique solution of the discrete problems, given by  
\begin{equation*}
\label{eq:aux2h}
A_h((\widehat{\bu}_h,\widehat{p}_h),(\bv_h,q_h))=c_h(\bu,\bv_h)\quad\forall (\bv_h,q_h)\in\bV_h\times\Q_h.
\end{equation*}
and remembering now, that  $(\bu_h,p_h)\in\bV_h\times\Q_h$ satisfies
\begin{equation*}
\label{eq:aux1h}
A_h((\bu_h,p_h),(\bv_h,q_h))=\lambda_h c_h(\bu_h,\bv_h)\quad\forall (\bv_h,q_h)\in\bV_h\times\Q_h,
\end{equation*}
Observe that the problem 
\begin{equation*}
A_h((\widehat{\bu}_h-\bu_h/\lambda_h,\widehat{p}_h-p_h/\lambda_h),(\bv_h,q_h))= c_h(\bu-\bu_h,\bv_h),
\end{equation*}
is well posed and its solution satisfies
\begin{equation*}
\left\|\widehat{\bu}_h-\frac{\bu_h}{\lambda_h} \right\|_{1,\O}+\left\|\widehat{p}_h-\frac{p_h}{\lambda_h} \right\|_{0,\O}\lesssim \|\bu-\bu_h\|_{0,\O},
\end{equation*}
where the hidden constant is independent of $h$.
Further, a straightforward derivation yields that 
\begin{multline*}
\|\widehat{\bu}-\widehat{\bu}_h\|_{1,\O}=\left\|\frac{\bu}{\lambda}-\widehat{\bu}_h\right\|_{1,\O}\leq \left\|\frac{\bu}{\lambda}-\frac{\bu_h}{\lambda} \right\|_{1,\O}
+\left\|\frac{\bu_h}{\lambda}-\frac{\bu_h}{\lambda_h} \right\|_{1,\O}+\left\|\frac{\bu_h}{\lambda_h}-\widehat{\bu}_h\right\|_{1,\O}\\
=\frac{\|\bu-\bu_h\|_{1,\O}}{\lambda}+\frac{|\lambda_h-\lambda|}{\lambda\lambda_h}\|\bu_h\|_{1,\O}+\left\|\frac{\bu_h}{\lambda_h}-\widehat{\bu}_h\right\|_{1,\O}\lesssim \|\bu-\bu_h\|_{1,\O}+|\lambda_h-\lambda|\|\bu_h\|_{1,\O}.
\end{multline*}
Proceeding in a similar way, we obtain
\begin{equation*}
\|\widehat{p}-\widehat{p}_h\|_{0,\O}\lesssim \|\bu-\bu_h\|_{1,\O}+\|p-p_h\|_{0,\O}+|\lambda_h-\lambda|\|p_h\|_{0,\O}.
\end{equation*}
Thus, we have obtained a control of the first two terms of \eqref{eq:aux0}. For the third term we use  
\begin{multline*}
\|\widehat{\bu}-\bPi_0\widehat{\bu}\|_{0,\O}\lesssim \|\widehat{\bu}-\widehat{\bu}_h\|_{0,\O}+\|\widehat{\bu}_h-\bPi_0\widehat{\bu}_h\|_{0,\O}\\\lesssim \|\widehat{\bu}-\widehat{\bu}_h\|_{1,\O}+\left\|\widehat{\bu}-\dfrac{\bu_h}{\lambda_h}\right\|_{1,\O}+\|\bu-\bPi_0\bu_h\|_{0,\O}\\
\lesssim \|\bu-\bu_h\|_{1,\O}+|\lambda_h-\lambda|\|\bu_h\|_{1,\O}+\|\bu-\bPi_0\bu_h\|_{0,\O}.
\end{multline*}
For the last term on the right hand side of \eqref{eq:aux0}, we proceed in a similar way to the previous case to obtain
\begin{equation*}
|\widehat{\bu}-\bPi\widehat{\bu}|_{1,\O}\lesssim \|\bu-\bu_h\|_{1,\O}+|\lambda_h-\lambda|\|\bu_h\|_{1,\O}+\|\bu-\bPi_0\bu_h\|_{0,\O}|\bu-\bPi\bu_h|_{1,\O}.
\end{equation*}
Replacing all the estimates obtained in \eqref{eq:aux0}, we have
\begin{multline*}
\|\bu-\bu_h\|_{0,\O}\lesssim h^s\left(\|\bu-\bu_h\|_{1,\O}+\|p-p_h\|_{0,\O}+|\lambda_h-\lambda|(\|\bu_h\|_{1,\O}+\|p_h\|_{0,\O})\right.\\
\left.+\|\bu-\bPi_0\bu_h\|_{0,\O}+|\bu-\bPi\bu_h|_{1,\O}\right).
\end{multline*}
The proof is concluded from the above estimate and Remark \ref{reamrk:ord2}.
\end{proof}
\noindent An important consequence of the spurious free feature of the proposed virtual  element method is that, for $h$ small enough, except 
for $\lambda_h$, the rest of the eigenvalues of \eqref{eq:weak_stokes_source_disc}are well separated from $\lambda$ (see  \cite{MR3647956}).
\begin{prop}
\label{separa_eig}
Let us enumerate the eigenvalues of problem \eqref{eq:weak_stokes_source_disc} in increasing order as follows: $0<\lambda_1\leq\cdots\lambda_i\leq\cdots$ and 
$0<\lambda_{h,1}\leq\cdots\lambda_{h,i}\leq\cdots$. Let us assume  that $\lambda_{\mathcal{J}}$ is a simple eigenvalue of \eqref{eq:weak_stokes_source_disc}. Then, there exists $h_0>0$ such that
\begin{equation*}
|\lambda_{\mathcal{J}}-\lambda_{h,i}|\geq\frac{1}{2}\min_{j\neq \mathcal{J}}|\lambda_j-\lambda_{\mathcal{J}}|\quad\forall i\leq \dim\bV_h,\,\,i\neq \mathcal{J},\quad \forall h<h_0.
\end{equation*}
\end{prop}

\section{A posteriori error analysis}
\label{sec:a_post}
The aim of this section is to introduce a suitable
residual-based error estimator for the Stokes
eigenvalue problem which is fully computable,
in the sense that  it depends only on quantities available from the VEM solution. Then, we will show its equivalence with the error.
Moreover, on the forthcoming analysis we will focus only on eigenvalues with simple multiplicity. With  this purpose, we introduce the following definitions and notations. For any polygon $\E\in \CT_h$, we denote by $\CE_{\E}$ the set of edges of $\E$
and 
$$\CE_h:=\bigcup_{\E\in\CT_h}\CE_{\E}.$$
We decompose $\CE_h=\CE_{\O}\cup\CE_{\partial\O}$,
where  $\CE_{\partial\O}:=\{\ell\in \CE_h:\ell\subset \partial\O\}$
and $\CE_{\O}:=\CE\backslash\CE_{\partial \O}$.
For each inner edge $\ell\in \CE_{\O}$ and for any  sufficiently smooth  function
$\bv$, we define the jump of its normal derivative on $\ell$ by
$$\left[\!\!\left[ \dfrac{\partial \bv}{\partial{ \boldsymbol{n}}}\right]\!\!\right]_\ell:=\nabla (\bv|_{\E})  \cdot \boldsymbol{n}_{\E}+\nabla ( \bv|_{\E'}) \cdot \boldsymbol{n}_{\E'} ,$$
where $\E$ and $\E'$ are  the two elements in $\CT_{h}$  sharing the
edge $\ell$ and $\boldsymbol{n}_{\E}$ and $\boldsymbol{n}_{\E'}$ are the respective outer unit normal vectors. As a consequence of the mesh regularity assumptions,
we have that each polygon $\E\in\CT_h$ admits a sub-triangulation $\CT_h^{\E}$
obtained by joining each vertex of $\E$ with the midpoint of the ball with respect
to which $\E$ is starred. Let $\CT_h:=\bigcup_{\E\in\CT_h}\CT_h^{\E}$.
Since we are also assuming \textbf{A2}, $\big\{\CT_h\big\}_{0 < h \leq 1}$
is a shape-regular family of triangulations of $\O$. 
Further, we introduce bubble functions on polygons as follows (see  \cite{MR3719046}).
 An interior bubble function $\psi_{\E}\in\H_0^1(\E)$ for a polygon $\E$
can be constructed piecewise as the sum of the  cubic
bubble functions  for each triangle of the
sub-triangulation $\CT_h^{\E}$ that attain the value 1 at the barycenter of the triangle. On the other hand, an edge bubble function $\psi_{\ell}$ for
$\ell\in\partial \E$ is a piecewise quadratic function
attaining the value 1 at the barycenter of $\ell$ and vanishing
on the polygons $\E\in\CT_h$  that do not contain $\ell$
on its boundary.

The following results which establish standard estimates
for bubble functions will be useful in what follows (see \cite{MR1885308,MR3059294}).

\begin{lemma}[Interior bubble functions]
\label{burbujainterior}
For any $\E\in \CT_h$, let $\psi_{\E}$ be the corresponding interior bubble function.
Then, there exists a constant $C>0$ 
independent of  $h_\E$ such that
\begin{align*}
C^{-1}\|q\|_{0,\E}^2&\leq \int_{\E}\psi_{\E} q^2\leq \|q\|_{0,\E}^2\qquad \forall q\in \mathbb{P}_k(\E),\\
C^{-1}\| q\|_{0,\E}&\leq \|\psi_{\E} q\|_{0,\E}+h_\E\|\nabla(\psi_{\E} q)\|_{0,\E}\leq C\|q\|_{0,\E}\qquad \forall q\in \mathbb{P}_k(\E).
\end{align*}
\end{lemma}
\begin{lemma}[Edge bubble functions]
\label{burbuja}
For any $\E\in \CT_h$ and $\ell\in\CE_{\E}$, let $\psi_{\ell}$
be the corresponding edge bubble function. Then, there exists
a constant $C>0$ independent of $h_\E$ such that
 \begin{equation*}
C^{-1}\|q\|_{0,\ell}^2\leq \int_{\ell}\psi_{\ell} q^2 \leq \|q\|_{0,\ell}^2\qquad
\forall q\in \mathbb{P}_k(\ell).
\end{equation*}
Moreover, for all $q\in\mathbb{P}_k(\ell)$, there exists an extension of  $q\in\mathbb{P}_k(\E)$ (again denoted by $q$) such that
 \begin{align*}
h_\E^{-1/2}\|\psi_{\ell} q\|_{0,\E}+h_\E^{1/2}\|\nabla(\psi_{\ell} q)\|_{0,\E}&\lesssim \|q\|_{0,\ell}.
\end{align*}
\end{lemma}
\begin{remark}
\label{extencion}
A possible way of extending $q$ from $\ell\in\CE_{\E}$ to $\E$
so that Lemma~\ref{burbuja} holds is  as follows:
first we extend $q$ to the straight line $L\supset\ell$ using the same polynomial function.
Then, we extend it to the whole plain through a constant
prolongation in the normal direction to $L$. Finally, we restrict  the latter to $\E$. 
\end{remark}
\noindent We define the edge residuals as follows:
\begin{equation}
\label{saltocal}
J_{\ell}:=\left\{\begin{array}{l}
\dfrac{1}{2}\left[\!\left[( p_h\mathbb{I}-\nabla\boldsymbol{\Pi}^{\E}\bu_h) \bn\right]\!\right]_{\ell} \qquad\;\; \;\forall \ell\in \CE_{\O},
\\[0.3cm]
\0  \qquad\qquad\qquad\qquad \qquad \qquad \quad \; \forall \ell\in\CE_{\partial\O},
\end{array}\right. 
\end{equation}
which is clearly computable.

Let us define the errors $\texttt{e}_{\bu}:=\bu-\bu_h$ and $\texttt{e}_p:=p-p_h$. Now we prove the following identity.
\begin{lemma}
\label{lmm:A_error}
Let $(\bv,q)\in \bV\times \Q$. Then we have the  following identity
\begin{multline*}
A((\texttt{e}_{\bu},\texttt{e}_p),(\bv,q))=\l c(\bu,\bv)-\l_h c(\bu_h,\bv)\\
-b(\bu_h,q)+\sum_{\E\in\CT_h}\l_hc^\E(\bu_h-\bPi_{0}^{\E}\bu_h,\bv)
-\sum_{\E\in\CT_h}a^\E(\bu_h-\boldsymbol{\Pi}^{\E}\bu_h,\bv)\\
+\sum_{\E\in\CT_h}\int_\E\left(\l_h\bPi_{0}^{\E}\bu_h+\nu\Delta\boldsymbol{\Pi}^{\E}\bu_h-\nabla p_h\right)\cdot \bv
+\sum_{\E\in\CT_h}\sum_{\ell\in\CE_{\E}}\int_\ell J_{\ell}v.
\end{multline*}
\end{lemma}
\begin{proof}
From the definition of $A(\cdot,\cdot)$, using the fact $(\lambda,(\bu,p))$ is a solution of \eqref{eq:eigen_A},  adding and subtracting  $\lambda_hc(\bu_h,\bv)$, and  integration by parts elementwise, we have
\begin{multline*}
A((\texttt{e}_{\bu},\texttt{e}_p),(\bv,q))=\lambda c(\bu,\bv)-A((\bu_h,p_h),(\bv,q))
=\lambda c(\bu,\bv)-\sum_{\E\in\CT_h}A^\E((\bu_h,p_h),(\bv,q))\\
=\lambda c(\bu,\bv)-\sum_{\E\in\CT_h}\left[a^\E(\bu_h,\bv)+b^\E(\bv,p_h)+b^\E(\bu_h,q)\right]\\
=\lambda c(\bu,\bv)-\sum_{\E\in\CT_h}\left[a^\E(\bu_h-\bPi^\E\bu_h,\bv)+a^{\E}(\bPi^\E\bu_h,\bv)+b^\E(\bv,p_h)+b^\E(\bu_h,q)\right]\\
=\lambda c(\bu,\bv)-\lambda_h c(\bu_h,\bv)-\sum_{\E\in\CT_h}a^\E(\bu_h-\bPi^\E\bu_h,\bv)\\
-\sum_{\E\in\CT_h}\left[a^{\E}(\bPi^\E\bu_h,\bv)+b^\E(\bv,p_h)-\lambda_hc^\E(\bu_h,\bv)\right]-\sum_{\E\in\CT_h}b^\E(\bu_h,q)\\
=\lambda c(\bu,\bv)-\lambda_h c(\bu_h,\bv)-\sum_{\E\in\CT_h}a^\E(\bu_h-\bPi^\E\bu_h,\bv)+\lambda_h\sum_{\E\in\CT_h} c^\E(\bu_h-\bPi_0^\E\bu_h,\bv)\\
-\sum_{\E\in\CT_h}\left[a^{\E}(\bPi^\E\bu_h,\bv)+b^\E(\bv,p_h)-\lambda_hc^\E(\bPi_0^\E\bu_h,\bv)\right]-\sum_{\E\in\CT_h}b^\E(\bu_h,q)\\
=\lambda c(\bu,\bv)-\lambda_h c(\bu_h,\bv)- b(\bu_h,q)-\sum_{\E\in\CT_h}a^\E(\bu_h-\bPi^\E\bu_h,\bv)+\lambda_h\sum_{\E\in\CT_h} c^\E(\bu_h-\bPi_0^\E\bu_h,\bv)\\
+\sum_{\E\in\CT_h}\left[\int_\E(\nu\Delta\bPi^\E\bu_h -\nabla p_h+\lambda_h\bPi_0^\E\bu_h)\cdot\bv\right]+\sum_{\E\in\CT_h}\sum_{\ell\in\CE_{\E}}\int_\ell J_{\ell}\bv, 
\end{multline*}
where $J_\ell$ has defined in \eqref{saltocal}. This concludes the proof.
\end{proof}
For all $\E\in\CT_h$, we introduce the following local terms and the local error indicator $\eta_\E$:
\begin{align*}
\Theta_\E^2&:=a_h^\E(\bu_h-\boldsymbol{\Pi}^{\E}\bu_h,\bu_h-\boldsymbol{\Pi}^{\E}\bu_h),\\
R_\E^2&:=h_\E^2\|\Upsilon_\E\|_{0,\E},\\
\eta_\E^2&:=\Theta_\E^2+R_\E^2+\sum_{\ell\in\CE_\E}h_\E\|J_{\ell}\|_{0,\ell}^2,
\end{align*}
where $\Upsilon_\E:=(\l_h\bPi_{0}^{\E}\bu_h+\nu\Delta\boldsymbol{\Pi}^{\E}\bu_h-\nabla p_h)|_\E$.
Now we are in a position to define the global error estimator by
\begin{equation}
\label{eq:global_eta}
\eta:=\left(\sum_{\E\in\CT_h}\eta_\E^2\right)^{1/2}.
\end{equation}

\noindent The next step is to prove that this estimator is reliable and efficient.
\subsection{Reliability}
As a first goal, we will prove upper bounds for the different error terms. We begin with the following result.
\begin{lemma}\label{lmm:bound_error_function}
The following estimate holds
\begin{equation*}
\|\bu-\bu_h\|_{1,\O}+\|p-p_h\|_{0,\O}\lesssim \eta+|\lambda-\lambda_h|+\|\bu-\bu_h\|_{0,\O},
\end{equation*}
where the hidden constant is independent of $h$.
\end{lemma}
\begin{proof}
From the definition of $A(\cdot,\cdot)$, algebraic manipulations leads to the following sequel of equalities
\begin{multline*}
\|\bu-\bu_h\|_{1,\O}+\|p-p_h\|_{0,\O}\leq A((\texttt{e}_{\bu},\texttt{e}_p),(\bv,q))=A((\texttt{e}_{\bu},\texttt{e}_p),(\bv-\bv_I,\mathcal{P}_{h}(q)))\\
+A((\texttt{e}_{\bu},\texttt{e}_p),(\bv_I,\mathcal{P}_{h}(q)))+A_h((\bu_h,p_h),(\bv_I,\mathcal{P}_{h}(q)))-A_h((\bu_h,p_h),(\bv_I,\mathcal{P}_{h}(q)))\\
=A((\texttt{e}_{\bu},\texttt{e}_p),(\bv-\bv_I,\mathcal{P}_{h}(q)))+\lambda c(\bu,\bv_I)-\lambda_hc_h(\bu_h,\bv_I)+a_h(\bu_h,\bv_I)-a(\bu_h\bv_I)\\
=\underbrace{\lambda c(\bu,\bv)-\lambda_h c(\bu_h,\bv)}_{T_1}+\underbrace{\lambda_h(c(\bu_h,\bv_I)-c_h(\bu_h,\bv_I))}_{T_2}+\underbrace{a_h(\bu_h,\bv_I)-a(\bu_h\bv_I)}_{T_3}\\
-\underbrace{ b(\bu_h,\mathcal{P}_{h}(q))}_{T_4}-\underbrace{\sum_{\E\in\CT_h}a^\E(\bu_h-\bPi^\E\bu_h,\bv-\bv_I)+\lambda_h\sum_{\E\in\CT_h} c^\E(\bu_h-\bPi_0^\E\bu_h,\bv-\bv_I)}_{T_5}\\
+\underbrace{\sum_{\E\in\CT_h}\left[\int_\E\Upsilon_\E\cdot(\bv-\bv_I)\right]+\sum_{\E\in\CT_h}\sum_{\ell\in\CE_{\E}}\int_\ell J_{\ell}(\bv-\bv_I)}_{T_6}.
\end{multline*}
Now we bound each of the contributions $T_i$, $i = 1,\ldots,6$.  Observe that for $T_1$, we have
\begin{equation}
\label{eq:boundT1}
T_1\leq \|\lambda\bu-\lambda_h\bu_h\|_{0,\O}\|\bv\|_{0,\O}\leq (|\lambda-\lambda_h|\|\bu\|_{0,\O}+|\lambda_h|\|\bu-\bu_h\|_{0,\O})\|\bv\|_{1,\O}.
\end{equation}
 \noindent For $T_2$, using the  stability property, we have
 \begin{equation}
 \label{eq:boundT2}
 T_2=\lambda_h\sum_{\E\in\CT_h}c^\E(\bu_h-\bPi_0^\E\bu_h,\bv_I) \lesssim \sum_{\E\in\CT_h}c^\E(\bu_h-\bPi_0^\E\bu_h,\bu_h-\bPi_0^\E\bu_h)^{1/2} \|\bv_I\|_{0,\E}.
 \end{equation}
 From the consistency,  we have for $T_3$
 \begin{multline}
 \label{eq:boundT3}
 T_3=\sum_{\E\in\CT_h}[a_h^\E(\bu_h-\bPi^\E\bu_h,\bv_I)-a^\E(\bu_h-\bPi^\E\bu_h,\bv_I)]\\
 \leq \sum_{\E\in\CT_h}a_h^\E(\bu_h-\bPi^\E\bu_h,\bu_h-\bPi^\E\bu_h)^{1/2} a_h(\bv_I,\bv_I)^{1/2}\\
 +\sum_{\E\in\CT_h}a^\E(\bu_h-\bPi^\E\bu_h,\bu_h-\bPi^\E\bu_h)^{1/2} a(\bv_I,\bv_I)^{1/2}\\
 \leq \left(\sum_{\E\in\CT_h}a_h^\E(\bu_h-\bPi^\E\bu_h,\bu_h-\bPi^\E\bu_h)\right)^{1/2}\|\bv\|_{1,\O}.
 \end{multline}
Upon using the definition of $\mathcal{P}_{h}$, and \eqref{eq:weak_stokes_eigen_disc}, we deduce that
  \begin{equation}
  \label{eq:boundT4}
  T_4=0.
   \end{equation}
 Proceeding as in the previous cases,  for $T_5$ it is easy to check that
 \begin{equation}
 \label{eq:boundT5}
 T_5\lesssim\left[\sum_{\E\in\CT_h}\left(a_h^\E(\bu_h-\bPi^\E\bu_h,\bu_h-\bPi^\E\bu_h)+\lambda_h c^\E(\bu_h-\bPi_0^\E\bu_h,\bu_h-\bPi_0^\E\bu_h)\right)\right]^{1/2}\|\bv\|_{1,\O}.
  \end{equation}
 For $T_6$ we have
 \begin{multline}
 \label{eq:boundT6}
 T_6\lesssim\sum_{\E\in\CT_h}\|\lambda_h\bPi_0^{\E}\bu_h+\nu\Delta\bPi^\E\bu_h-\nabla p_h\|_{0,\E}\|\bv-\bv_I\|_{0,\E}+\sum_{\ell\in\mathcal{E}_\E}\int_{\ell} J_\ell(\bv-\bv_I)\\
 \lesssim \sum_{\E\in\CT_h}\left[h_\E\|\lambda_h\bPi_0^{\E}\bu_h+\nu\Delta\bPi^\E\bu_h-\nabla p_h\|_{0,\E}\|\bv\|_{1,\E}+\sum_{\ell\in\mathcal{E}_\E}h_\E^{1/2}\|J_\ell\|_{0,\ell}\|\bv\|_{1,\E} \right]\\
 \lesssim \left\{\sum_{\E\in\CT_h}[h_\E^2\|\lambda_h\bPi_0^{\E}\bu_h+\nu\Delta\bPi^\E\bu_h-\nabla p_h\|_{0,\E}^2+\sum_{\ell\in\mathcal{E}_\E}h_\E\|J_\ell\|_{0,\ell}^2] \right\}^{1/2}\|\bv\|_{1,\O}.
 \end{multline}
Finally, we note that 
 \begin{multline}
 \label{eq:bound_G}
\sum_{\E\in\CT_h}c^\E(\bu_h-\bPi_0^\E\bu_h,\bu_h-\bPi_0^\E\bu_h)= \sum_{\E\in\CT_h}\|\bu_h-\bPi_0^\E\bu_h\|_{0,\E}^2\leq \sum_{\E\in\CT_h}\|\bu_h-\bPi^\E\bu_h\|_{0,\E}^2\\
 \lesssim \sum_{\E\in\CT_h}|\bu_h-\bPi^\E\bu_h|_{1,\E}^2\lesssim \sum_{\E\in\CT_h}a_h^\E(\bu_h-\bPi^\E\bu_h,\bu_h-\bPi^\E\bu_h),
  \end{multline}
  where we have employed the fact that the projection operator $\bPi^\E$ is invariant on polynomial and standard approximation property of the projector. 
 Hence, gathering \eqref{eq:boundT1}, \eqref{eq:boundT2}, \eqref{eq:boundT3}, \eqref{eq:boundT4}, \eqref{eq:boundT5}, \eqref{eq:boundT6} and \eqref{eq:bound_G} we conclude the proof.
\end{proof}

\begin{lemma}\label{lmm:estimateupL2}
The following estimate holds
\begin{equation*}
\|\bu-\bu_h\|_{1,\O}+\|p-p_h\|_{0,\O}+\|\bu-\bPi_0\bu_h\|_{0,\O}+\|\bu-\bPi\bu_h\|_{1,h}\lesssim \eta+|\lambda-\lambda_h|+\|\bu-\bu_h\|_{0,\O},
\end{equation*}
where the hidden constant is  independent of $h$.
\end{lemma}
\begin{proof}
For each polygon $\E\in\CT_h$, we have
\begin{equation*}
\|\bu-\bPi_0^\E\bu_h\|_{0,\E}+\|\bu-\bPi^\E\bu_h\|_{1,\E}\leq \|\bu-\bu_h\|_{1,\E}+\|\bu_h-\bPi_0^\E\bu_h\|_{0,\E}+\|\bu_h-\bPi^\E\bu_h\|_{1,\E}.
 \end{equation*}
The result follows, summing over all polygons and using that
 \begin{equation}\label{eq:cotasupstba}\|\bu_h-\bPi_0^\E\bu_h\|_{0,\E}+\|\bu_h-\bPi^\E\bu_h\|_{1,\E}\lesssim \Theta_\E^2\leq \eta_\E^2,
 \end{equation} together with Lemma \ref{lmm:bound_error_function}.
\end{proof}
\begin{lemma}\label{lmm:bound_error_lambdas}
The following estimate holds
\begin{equation*}
|\lambda-\lambda_h|\lesssim \eta^2+|\lambda-\lambda_h|^2+\|\bu-\bu_h\|_{0,\O},
\end{equation*}
where the hidden constant is independent of $h$.
\end{lemma}
\begin{proof}
The proof  follows from Remark \ref{reamrk:ord2} and the Lemma \ref{lmm:estimateupL2}.
\end{proof}
\begin{corollary}
There exists $h_0$ such that, for all $h<h_0$, there holds
\begin{equation}\label{eq:Reliability1}
\|\bu-\bu_h\|_{1,\O}+\|p-p_h\|_{0,\O}+\|\bu-\bPi_0\bu_h\|_{0,\O}+\|\bu-\bPi\bu_h\|_{1,h}\lesssim \eta,
\end{equation}
and 
\begin{equation}
\label{eq:Reliability2}
|\lambda-\lambda_h|\lesssim \eta^2.
\end{equation}
\end{corollary}
\begin{proof}
Let us define $$\mathcal{Z}:=\|\bu-\bu_h\|_{1,\O}+\|p-p_h\|_{0,\O}+\|\bu-\bPi_0\bu_h\|_{0,\O}\\+\|\bu-\bPi\bu_h\|_{1,h}.$$ Then,  from Lemmas \ref{The:L2up} and \ref{lmm:estimateupL2} we have 
\begin{equation*}
\mathcal{Z}\lesssim \eta+h^s\mathcal{Z}.
\end{equation*}
Therefore, it is easy to check that there exists $h_0 > 0$ such that for all $h<h_0$ \eqref{eq:Reliability1} holds true.
In order to prove \eqref{eq:Reliability2}, we note that from Lemma \ref{lmm:bound_error_lambdas}, Remark \ref{reamrk:ord2}, Lemma \ref{The:L2up} and \eqref{eq:Reliability1}, we have 
\begin{equation*}
(1-h^r)|\lambda-\lambda_h|\lesssim \mathcal{Z}^2+\|\bu-\bu_h\|_{0,\O}^2\lesssim  (1+h^s)\eta^2.
\end{equation*}
Hence, for $h$ small enough, \eqref{eq:Reliability2} holds true, concluding the proof.
\end{proof}
\subsection{Efficiency}
In what follows we will provide the efficiency bound for $\eta$. The forthcoming results are based in the
classic technique of bubble functions.
\begin{lemma}
\label{lmm:bound_R_1}
The following estimate holds
\begin{equation*}
\|\Upsilon_\E\|_{0,\E}\lesssim h_\E^{-1}\left(\Theta_\E+|\bu-\bu_h|_{1,\E}+\|p-p_h\|_{0,\E} +h_\E(\|\bu_h-\Pi_0^\E\bu_h\|_{0,\E}+\|\lambda\bu_h-\lambda_h\bu_h\|_{0,\E})\right),
\end{equation*}
where the hidden constant is independent of $h_\E$.
\end{lemma}
\begin{proof}
Let $\psi_\E$ be as in Lemma \ref{burbujainterior}. Let us define the function $\bv:=\Upsilon_\E\psi_\E$. Notice that $\bv$ vanishes in the boundary of $\E$ and also, may be extended by zero to the whole domain $\O$, implying that $\bv\in[\H^1(\O)]^2$. Then, from Lemma \ref{lmm:A_error}, on each $\E\in\CT_h$, we have
\begin{multline}
\label{eq:indentity_bubble}
A((\texttt{e}_u,\texttt{e}_p),(\bv,q))=\lambda c^{\E}(\bu,\Upsilon_{\E}\psi_\E)-\lambda_hc^\E(\bu_h,\Upsilon_{\E}\psi_\E)\\
-b^\E(\bu_h,q)- a^\E(\bu_h-\bPi^\E\bu_h,\Upsilon_{\E}\psi_\E)
+ \lambda_h c^\E(\bu_h-\bPi_0^\E\bu_h,\Upsilon_{\E}\psi_\E)
+ \int_\E\Upsilon_\E^2\psi_{\E}.
\end{multline}
Now, from \eqref{eq:indentity_bubble}, the continuous incompressibility condition,  and Lemma \ref{burbujainterior} we obtain
\begin{multline*}
\|\Upsilon_\E\|_{0,\E}^2\lesssim \int_\E\psi_\E \Upsilon_\E^2=A^{\E}((\texttt{e}_u,\texttt{e}_p),(\Upsilon_\E\psi_\E,q))+a^\E(\bu_h-\bPi^\E\bu_h,\Upsilon_\E\psi_\E)\\
-\lambda_h c^\E(\bu_h-\bPi_0^\E\bu_h,\Upsilon_\E\psi_\E)+b^\E(\bu_h,q)-[\lambda c^{\E}(\bu,\Upsilon_\E\psi_\E)-\lambda_hc^\E(\bu_h,\Upsilon_\E\psi_\E)]\\
=a^\E(\texttt{e}_{\bu},\Upsilon_\E\psi_\E)+b^\E(\Upsilon_\E\psi_\E,\texttt{e}_p)-[\lambda c^{\E}(\bu,\Upsilon_\E\psi_\E)-\lambda_hc^\E(\bu_h,\Upsilon_\E\psi_\E)]\\
-a^\E(\bu_h-\bPi^\E\bu_h,\Upsilon_\E\psi_\E)- \lambda_h c^\E(\bu_h-\bPi_0^\E\bu_h,\Upsilon_\E\psi_\E)\\
\lesssim h_\E^{-1}(|\texttt{e}_{\bu}|_{1,\E} + \|\texttt{e}_p\|_{0,\E}+|\bu_h-\bPi^\E\bu_h|_{1,\E} \\
+h_\E(\|\bu_h-\bPi_0^\E\bu_h\|_{0,\E}+\|\lambda\bu-\lambda_h\bu_h\|_{0,\E})\|\Upsilon_\E\|_{0,\E}\\
\lesssim h_\E^{-1}(|\texttt{e}_{\bu}|_{1,\E} + \|\texttt{e}_p\|_{0,\E}+\Theta_\E+h_\E(\|\bu_h-\Pi_0^\E\bu_h\|_{0,\E}+\|\lambda\bu_h-\lambda_h\bu_h\|_{0,\E})),
\end{multline*}
where for the last estimate we used \eqref{eq:cotasupstba}. Then we concluding the proof.
\end{proof}
We now need a control of $\Theta_\E$.
\begin{lemma}
\label{lmm:bound:theta}
The following estimate holds
\begin{equation*}
\Theta_\E\lesssim \|\bu-\bu_h\|_{1,\E}+|\bu-\bPi^\E\bu_h|_{1,\E},
\end{equation*}
where the hidden constant is independent of $h_\E$.
\end{lemma}
\begin{proof}
According to \cite[Remark 3.2]{MR3340705}, the following identity  holds $$a_h^\E(\bu_h-\bPi^\E\bu_h,\bu_h-\bPi^\E\bu_h)=S^\E(\bu_h-\bPi^\E\bu_h,\bu_h-\bPi^\E\bu_h).$$
Then, from the definitions of $\Theta_\E$ and triangle inequality  we have
\begin{multline*}
\Theta_\E^2\lesssim |\bu_h-\bPi^\E\bu_h|_{1,\E}^2\leq  |\bu_h-\bu|_{1,\E}^2+|\bu-\bPi^\E\bu_h|_{1,\E}^2
\leq  \|\bu_h-\bu\|_{1,\E}^2+|\bu-\bPi^\E\bu_h|_{1,\E}^2,
\end{multline*}
which concludes the proof.
\end{proof}
Now our aim is to control the term associated to the jumps.
\begin{lemma}\label{lmm:first_jump}
The following estimate hold
\begin{multline*}
h_{\E}^{1/2}\|J_\ell\|_{0,\ell}\lesssim \sum_{\E'\in\omega_{\ell}}\left(\|\texttt{e}_{\bu}\|_{1,\E'}+\|\texttt{e}_p\|_{0,\E'}+|\bu-\bPi^{\E'}\bu_h|_{1,\E'}\right.\\
\left.+h_{\E'}\|\lambda\bu-\lambda_h\bu_h\|_{0,\E'}+h_{\E'}\|\bu_h-\bPi_0^{\E'}\bu_h\|_{0,\E'}\right),
\end{multline*}
where $\omega_{\ell}:=\{\E'\in\CT_h: \ell\in \mathcal{E}_{\E'}\}$ and the hidden constant is independent of $h_{\E}$.
\end{lemma}
\begin{proof}
Let $\ell\in\mathcal{E}_K\cap\mathcal{E}_{\O}$. Let us define $\bv:=J_\ell\psi_\ell$ where $\psi_\ell$ is the bubble function satisfying the properties of Lemma \ref{burbuja}. Observe that $\bv$ may be extended by zero to the whole domain $\O$. For simplicity, we denote by $\bv$ such an extension. Observe that $\bv\in \bV$.

Invoking Lemma \ref{lmm:A_error}, the following identity holds
\begin{multline*}
A((\texttt{e}_{\bu},\texttt{e}_p),(\bv,q))=\lambda c(\bu,J_\ell\psi_\ell)-\lambda_hc(\bu_h,J_\ell\psi_\ell)-b(\bu_h,q)\\
+\sum_{\E'\in \omega_{\ell}}\left(\lambda_hc^{\E'}(\bu_h-\bPi_0^\E\bu_h,J_\ell\psi_\ell)-a^{\E'}(\bu_h-\bPi^\E\bu_h,J_\ell\psi_\ell)\right)\\
+\sum_{\E'\in \omega_{\ell}}\left(\int_{\E'}\Upsilon_{\E'}\cdot(J_\ell\psi_\ell)+\int_\ell J_\ell^2\psi_l\right).
\end{multline*}
Then we have
\begin{multline}\label{eq:estimate_jumpup}
\|J_{\ell}\|_{0,\ell}^2\lesssim\int_{\ell}J_{\ell}^2\psi_{\ell}=\sum_{\E'\in \omega_{\ell}}(a^{\E'}(\texttt{e}_{\bu},J_{\ell}\psi_{\ell})+b^{\E'}(J_{\ell}\psi_{\ell},\texttt{e}_p)+b(\texttt{e}_{\bu},q)+b(\bu_h,q)\\
-[\lambda c(\bu,J_{\ell}\psi_{\ell})-\lambda_hc(\bu_h,J_{\ell}\psi_{\ell})]-\sum_{\E'\in \omega_{\ell}}\lambda_h c^{\E'}(\bu_h-\bPi_0^{\E}\bu_h,J_{\ell}\psi_{\ell})\\
+\sum_{\E'\in \omega_{\ell}}a^{\E'}(\bu_h-\bPi^{\E}\bu_h,J_{\ell}\psi_{\ell})-\sum_{\E'\in \omega_{\ell}}\int_{\E'}\Upsilon_{\E'}\cdot(J_{\ell}\psi_{\ell})\\
\lesssim |\texttt{e}_{\bu}|_{1,\E'}|J_{\ell}\psi_{\ell}|_{1,\E'}+\|\texttt{e}_p\|_{0,\E'}|J_{\ell}\psi_{\ell}|_{1,\E'}+\|\lambda\bu-\lambda_h\bu_h\|_{0,\E'}\|J_{\ell}\psi_{\ell}\|_{0,\E'}\\
+\|\bu_h-\bPi_0^{\E'}\bu_h\|_{0,\E'}\|J_{\ell}\psi_{\ell}\|_{0,\E'}+|\bu_h-\bPi^{\E'}\bu_h|_{1,{\E}'}|J_{\ell}\psi_{\ell}|_{1,\E'}+\|\Upsilon_{\E'}\|_{0,\E'}\|J_{\ell}\psi_{\ell}\|_{0,\E'},
\end{multline}
where for the last estimate we have used that $b(\bu,q)=0$ for all $q\in\Q$. 

On the other hand notice that from Lemma \ref{burbuja} we have
\begin{multline*}
|\texttt{e}_{\bu}|_{1,\E'}|J_{\ell}\psi_{\ell}|_{1,\E'}+\|\texttt{e}_p\|_{0,\E'}|J_{\ell}\psi_{\ell}|_{1,\E'}\lesssim (|\texttt{e}_{\bu}|_{1,\E'}+\|\texttt{e}_p\|_{0,\E'})h_{\E'}^{-1/2}\|J_{\ell}\|_{0,\ell};\\
\|\lambda\bu-\lambda_h\bu_h\|_{0,\E'}\|J_{\ell}\psi_{\ell}\|_{0,\E'}\lesssim h_{\E'}^{1/2}\|\lambda\bu-\lambda_h\bu_h\|_{0,\E'}\|J_{\ell}\|_{0,\ell};\\
\|\bu_h-\bPi_0^{\E'}\bu_h\|_{0,\E'}\|J_{\ell}\psi_{\ell}\|_{0,\E'}+|\bu_h-\bPi^{\E'}\bu_h|_{1,\E'}|J_{\ell}\psi_{\ell}|_{1,\E'}\\
\lesssim (h_{\E'}\|\bu_h-\Pi_0^{\E'}\bu_h\|_{0,\E'}+\Theta_{\E'})h_{\E}^{-1/2}\|J_{\ell}\|_{0,\ell};\\
\|\Upsilon_{\E'}\|_{0,\E'}\|J_{\ell}\psi_{\ell}\|_{0,\E'}\lesssim h_{\E'}^{1/2}\|\Upsilon_{\E'}\|_{0,\E'}\|J_{\ell}\|_{0,\ell}\\
\lesssim h_{\E'}^{-1/2}(|\bu-\bu_h|_{1,\E'}+\|p-p_h\|_{0,\E'}+\Theta_{\E'}+h_{\E'}\|\lambda\bu_h-\lambda_h\bu_h\|_{0,\E'}).
\end{multline*}
The proof is concluded by combining \eqref{eq:estimate_jumpup} and the above estimates.


\end{proof}
\noindent Now we are in a position to prove the efficiency of our local error indicator $\eta_\E.$
\begin{lemma}[Local efficiency]
\label{lmm:upper_bound1}
The following upper bound for the local indicator $\eta_\E$ holds
\begin{multline*}
\eta_\E^2 \lesssim\sum_{\E'\in\omega_\E}\left(\|\bu-\bu_h\|_{1,\E'}^2+\|p-p_h\|_{0,\E'}^2+|\bu-\bPi^{\E'}\bu_h|_{1,\E'}^2\right.\\
\left.+\|\bu-\bPi_0^{\E'}\bu_h\|_{0,\E'}^2+h_{\E'}^2\|\lambda\bu-\lambda_h\bu_h\|_{0,\E'}\right),
\end{multline*}
where the $\omega_{\E}:=\{\E'\in\CT_h :\E' \,\text{and $\E$ share an edge}\}$.
\end{lemma}
\begin{proof}
The result follows immediately from Lemmas \ref{lmm:bound_R_1}--\ref{lmm:first_jump}.
\end{proof}
The next step is to establish that the term $h_{\E}\|\lambda \bu-\lambda\bu_h\|_{0,\E'}$ appearing in the above estimation is asymptotically negligible for the overall estimator.
\begin{corollary}
The following estimate holds
\begin{equation*}
\eta^2\lesssim \|\bu-\bu_h\|_{1,\O}^2+\|p-p_h\|_{0,\O}^2+|\bu-\bPi\bu_h|_{1,h}^2+\|\bu-\bPi_0\bu_h\|_{0,\O}^2,
\end{equation*}
where the hidden constant is independent of $h$.
\end{corollary}
\begin{proof}
From Lemma \ref{lmm:upper_bound1} we have 
\begin{multline*}
\eta^2\lesssim \|\bu-\bu_h\|_{1,\O}^2+\|p-p_h\|_{0,\O}^2+|\bu-\bPi\bu_h|_{1,h}^2
+\|\bu-\bPi_0\bu_h\|_{0,\O}^2+h^2\|\lambda \bu-\lambda\bu_h\|_{0,\O}.
\end{multline*}
Now, for the last term, taking into account that $\|\bu_h\|_{0,\O}=1$, the following estimate holds
\begin{equation*}
\|\lambda\bu-\lambda_h\bu\|_{0,\O}^2\leq 2\lambda^2\|\bu-\bu_h\|_{0,\O}^2+2|\lambda-\lambda_h|^2\lesssim \|\bu-\bu_h\|_{1,\O}^2+2|\lambda-\lambda_h|^2.
\end{equation*}
Invoking Remark \ref{reamrk:ord2}, we have
\begin{equation*}
|\lambda-\lambda_h|^2\leq(|\lambda|+|\lambda_h|)|\lambda-\lambda_h| \|\bu-\bu_h\|_{1,\O}^2+\|p-p_h\|_{0,\O}^2
+ |\bu-\bPi\bu_h|_{1,h}^2+\|\bu-\bPi_0\bu_h\|_{0,\O}^2,
\end{equation*}
and hence
\begin{equation*}
\eta^2\lesssim \|\bu-\bu_h\|_{1,\O}^2+\|p-p_h\|_{0,\O}^2+|\bu-\bPi\bu_h|_{1,h}^2+\|\bu-\bPi_0\bu_h\|_{0,\O}^2,
\end{equation*}
which concludes the proof.

\end{proof}

\section{Numerical experiments}
\label{sec:numerics}
In this section, we report some numerical experiments to validate the theoretical estimates and performance of the proposed scheme on convex, and non-convex domain generally appeared in practical application. In order to remain simple, along this section we will consider as kinematic viscosity $\nu=1$.  We propose first three numerical tests to verify the optimal order of convergence, i.e., $\mathcal{O}(h^{2r})$ of the eigenvalues on convex domain and robustness of the proposed scheme on a general shaped domain consists of holes and re-entrant angles. Additionally, we assess the influence of the stabilizer in the computation of eigenvalues in the fourth example. Moreover, we have considered a non-convex domain ($\L$ shaped) with singularity as a final benchmark example to judge the estimators experimentally. The presence of singularity affects the convergence order of the eigenvalues (see \cite{MR2473688,M1402959} for instance). In view of this fact, we have proposed a suitable residual based estimator to retrieve the optimal order of convergence by refining the area with spurious oscillation. Further, we discretize the computational domain with different type of elements such as square, voronoi, nonconvex, and remapped polygons as shown in Figure~\ref{FIG:meshes}. Generally, by employing the standard basis of the virtual space, the discrete bilinear forms~\eqref{eq:weak_stokes_eigen_disc} leads to generalized matrix eigenvalue problem of the layout
\begin{equation*}
\AT  \bold{U}_h=\lambda_h \BT \bold{U}_h,
\end{equation*}
where $\bold{U}_h:=(\bu_h,p_h)^{\texttt{t}}$.  We remark that matrix $\AT$ is symmetric and positive definite whereas $\BT$ is symmetric and semipositive definite. Therefore, we have solved using the  MATLAB command \texttt{eigs}, the following equivalent problem
\begin{equation*}
\BT \bold{U}_h=\frac{1}{\lambda_h} \AT \bold{U}_h.
\end{equation*}
For better presentation of the article, we denote by $N$ the number of elements on each side of the domain.

\begin{figure}[H]
	\begin{center}
		\begin{minipage}{13cm}
			\centering\includegraphics[height=3cm, width=3cm]{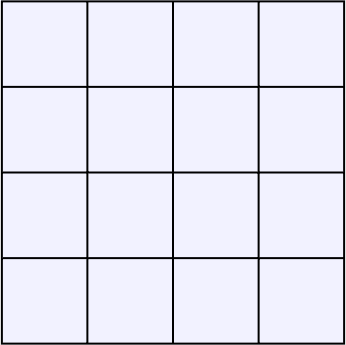}\hspace{.5 cm}
			\centering\includegraphics[height=3cm, width=3cm]{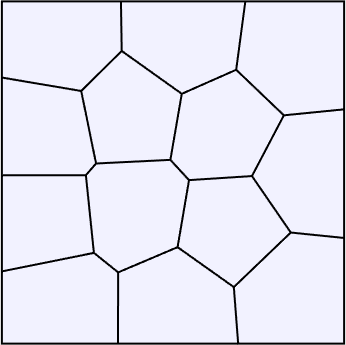}\hspace{.5 cm}
                          \centering\includegraphics[height=3cm, width=3cm]{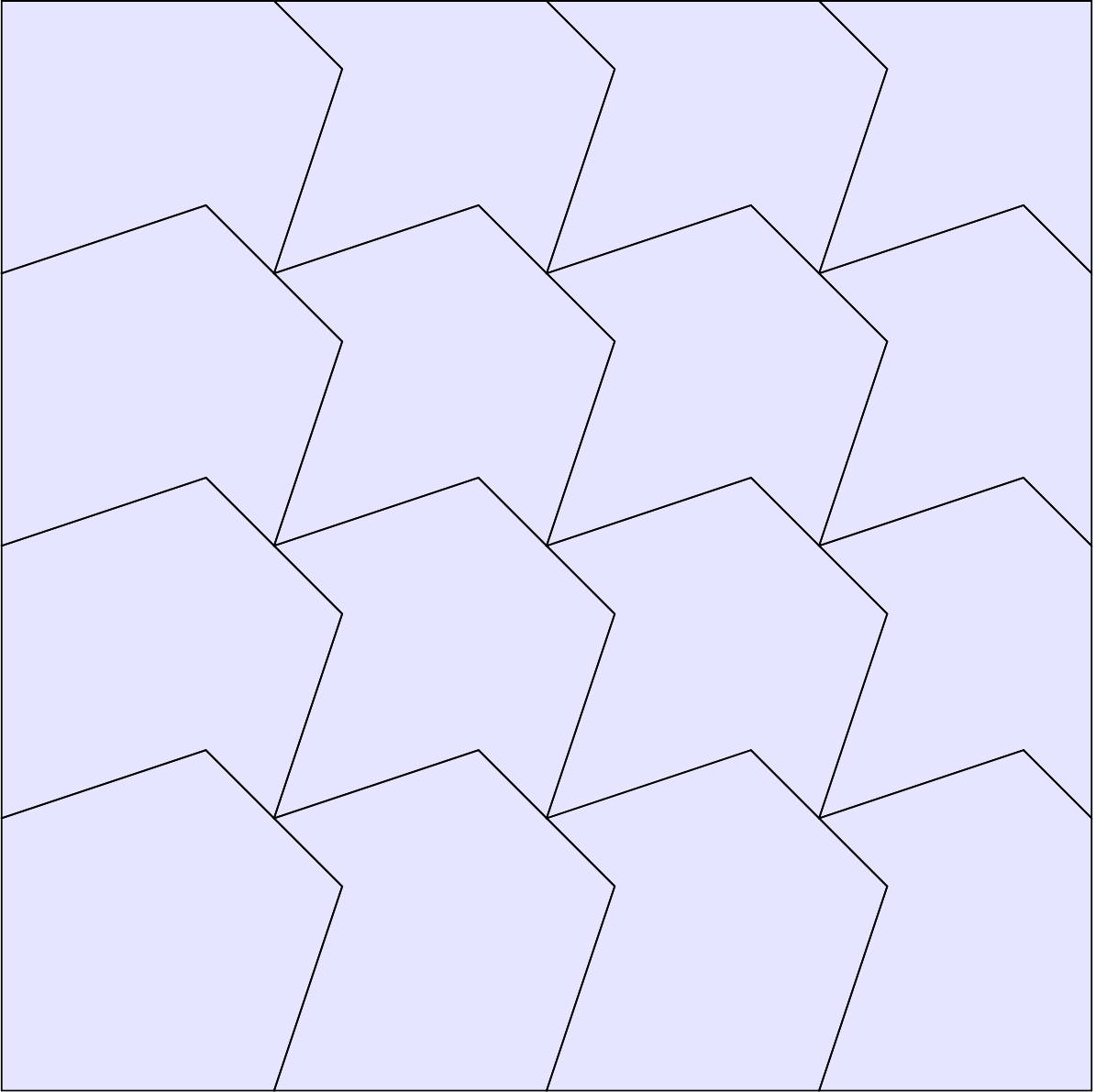}\\
                          \vspace{.5 cm}
                          \centering\includegraphics[height=3cm, width=3cm]{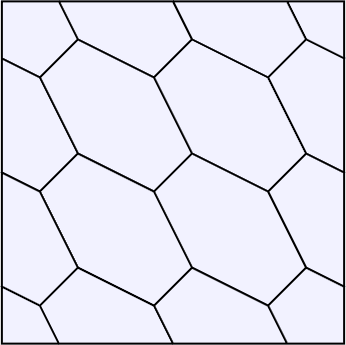} \hspace{.5 cm}
                          \centering\includegraphics[height=3cm, width=3cm]{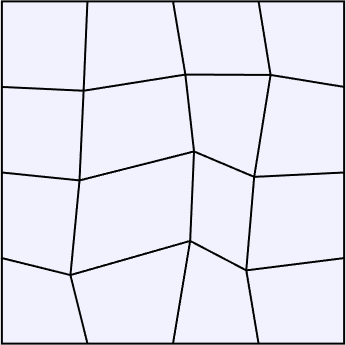}
                   \end{minipage}
		\caption{Schematic diagram of sample meshes. From top left to bottom right: $\CT_h^1$, $\CT_h^2$, $\CT_h^3$, $\CT_h^4$, and $\CT_h^5$ respectively, with $N=4.$ 
}
		\label{FIG:meshes}
	\end{center}
\end{figure}

%

\subsection{Clamped square.} 
In this test, we have carried out the experiments by considering the computational domain as $\O:=(-1,1)^2$, with clamped boundary conditions. The zero mean condition on pressure variable, i.e., $\int_{\O} p=0$ has been 
incorporated to the system as a Lagrange multiplier. We have computed the  four lowest eigenvalues, which are reported  in Table~\ref{TABLA:clamped:plate} for different meshes. Further, we emphasize that our domain is convex in shape, and consequently the eigenfunctions are smooth enough which leads to optimal quadratic order of convergence ({\it Lowest order VEM space}) for the eigenvalues as reported in Table~\ref{TABLA:clamped:plate}. The  computed pressure and velocity fields associated with third lowest eigenvalues are displayed in Figure~\ref{FIG:clampedPlate}. 


\begin{table}[H]
\begin{center}
\caption{Test 1: Lowest eigenvalues $\l_{h}^{(i)}$, $1\le
i\le4$ of clamped square on different meshes}
\begin{tabular}{|c||c||ccc||c||c||c|c|c }
\hline
$\CT_h$   & $\l_{h}^{(i)}$  & $N=16$ & $N=32$ & $N=64$ &
Order & Extr.&\cite{MR2473688} \\
\hline
\hline
& $\l_{h}^{(1)}$  & 13.0937  & 13.0883  & 13.0870 &   2.05  & 13.0865&13.086\\
& $\l_{h}^{(2)}$   & 22.9910 &  23.0219 &  23.0300 &   1.95 &  23.0327&23.031\\
$\CT_h^1$ & $\l_{h}^{(3)}$   & 22.9910  & 23.0219 &  23.0300  & 1.95  & 23.0327&23.031\\
& $\l_{h}^{(4)}$   & 31.7954  & 32.0009 &  32.0452   & 2.22  & 32.0572&32.053\\
\hline
\hline
& $\l_{h}^{(1)}$ &  13.0749   &13.0830 &  13.0851   & 1.95  & 13.0858&13.086\\
& $\l_{h}^{(2)}$  &  22.9788  & 23.0185 &  23.0284   & 2.02  & 23.0316&23.031\\
$\CT_h^2$& $\l_{h}^{(3)}$     &22.9832  & 23.0202   &23.0287   & 2.14  & 23.0311&23.031\\
& $\l_{h}^{(4)}$    & 31.9339   &32.0161 &  32.0364   & 2.02 & 32.0431 &32.053\\
  \hline
  \hline
& $\l_{h}^{(1)}$ &  13.0780   &13.0835  & 13.0847 &   2.17  & 13.0850&13.086\\
& $\l_{h}^{(2)}$  &  22.9708  & 23.0195&   23.0301 &   2.20 &  23.0330&23.031\\
$\CT_h^3$& $\l_{h}^{(3)}$     & 22.9850 &  23.0212 &  23.0301 &   2.02  & 23.0330&23.031\\
& $\l_{h}^{(4)}$   & 31.8912  & 32.0153 &  32.0420  &  2.21  & 32.0495&32.053\\
\hline
\hline
& $\l_{h}^{(1)}$  & 13.5218   &13.1750 &  13.0993   & 2.19 & 13.0778&13.086\\
& $\l_{h}^{(2)}$   &  24.2011  & 23.2899 &  23.0791  &  2.11 &  23.0154&23.031\\
$\CT_h^4$ & $\l_{h}^{(3)}$  &  24.2136  & 23.2910 &  23.0499  &  1.94  & 22.9653 &32.031\\
 & $\l_{h}^{(4)}$  &  32.2218  & 32.1257 &  32.1011  &  1.97  & 32.0927&32.053\\
\hline
 \end{tabular}
\label{TABLA:clamped:plate}
\end{center}
\end{table}
 Moreover, 
the computed extrapolated eigenvalues of our method are in accordance with those available in the literature (cf. \cite{MR2473688}) irrespective of the shape of the elements used for domain discretization. This experiment asserts that our method is a challenging competitor with the existing techniques available in the literature.

\begin{figure}[H]
	\begin{center}
		\begin{minipage}{11cm}
			 \centering\includegraphics[height=3.8cm, width=4.4cm]{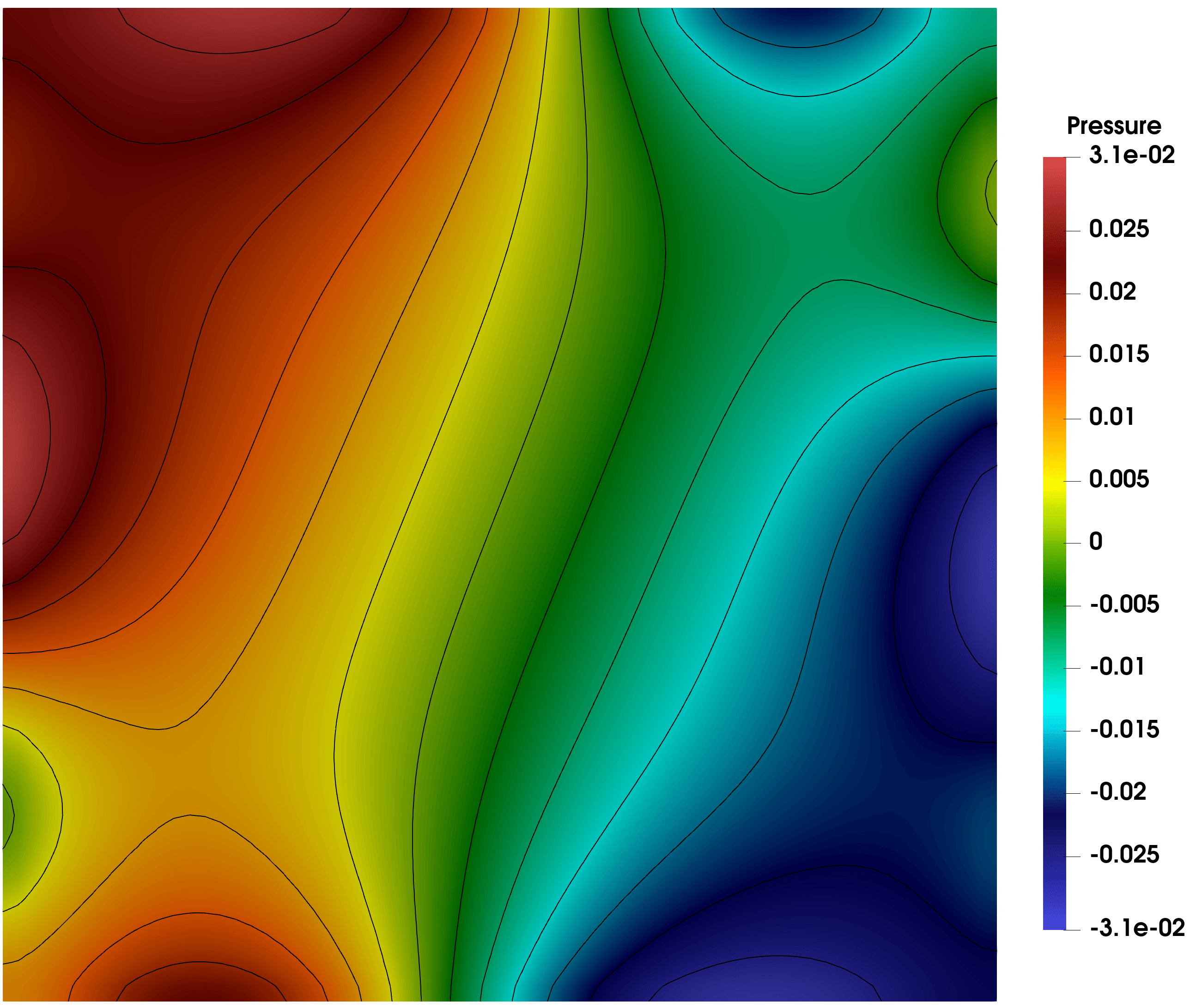}
                          \centering\includegraphics[height=3.8cm, width=4.4cm]{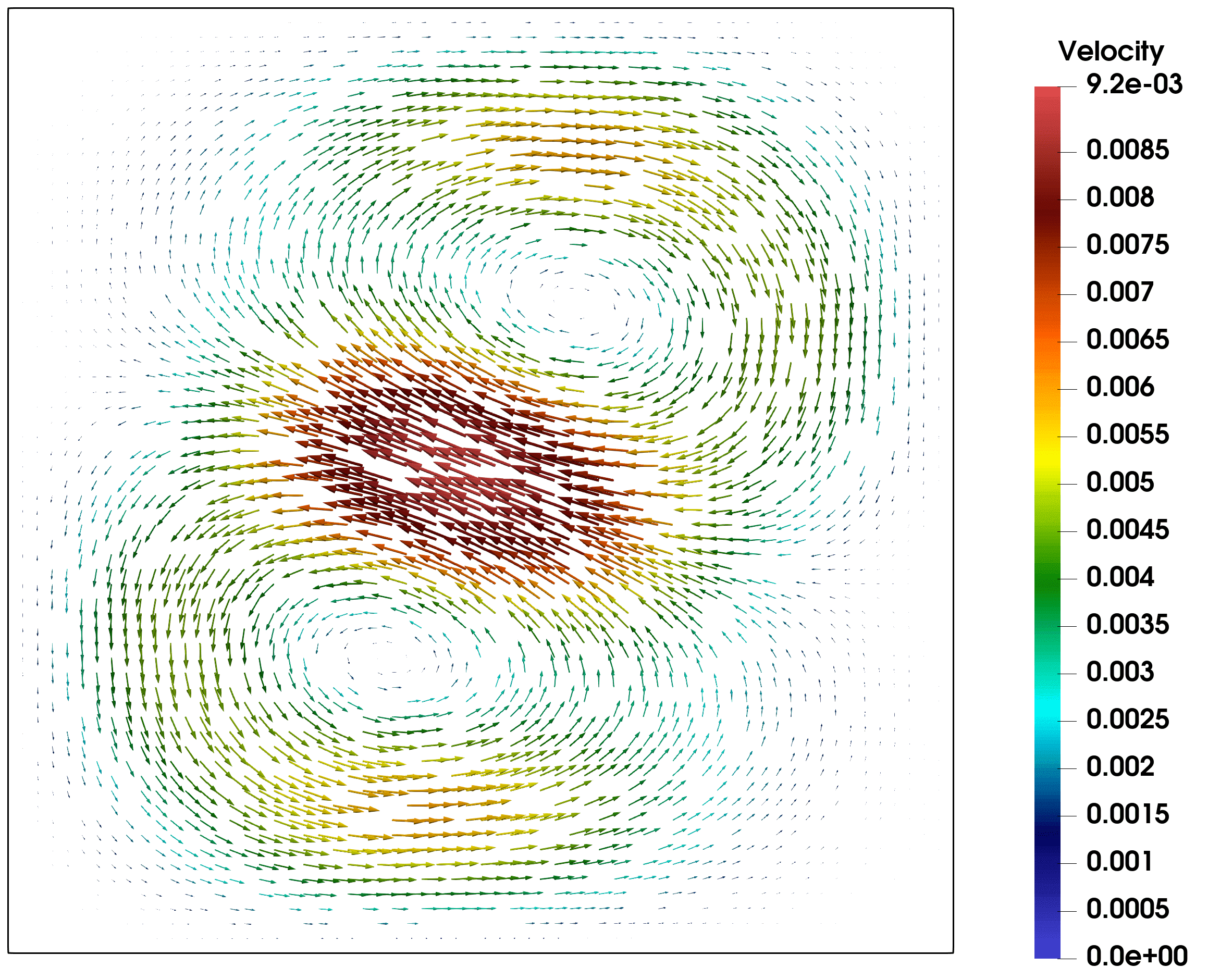}\\
                    \end{minipage}
		\caption{ ``Snapshots'' of pressure and velocity fields corresponding to the  third  lowest eigenvalue with $\CT_{h}^{1}$ and $N=64$.}
		\label{FIG:clampedPlate}
	\end{center}
\end{figure}

\subsection{Circular Domain} For this example we consider as computational domain the unit circular plate, $\overline{\Omega}:=\{(x,y)\in \mathbb{R}^2: x^2+y^2\leq 1 \}$. 
 An example of the meshes that we consider for this test is $ \CT_h^2$ presented in Figure~\ref{FIG:meshes}. 

The purpose of this example is to observe the performance of the proposed method in a context that goes beyond of the developed theory, where a curved domain is considered.
Further, we discretize the domain with voronoi mesh and the curve boundary is approximated by straight-lines. To avoid the variational crime due to approximation of the curve boundary by straightline, we have considered finer mesh for the computation of spectrum. Additionally, we highlight that we have employed clamped boundary condition for the test. The computed spectrum is reported in Table~\ref{TABLA:clamped}. Let us remark that our computed eigenvalues are similar with those obtained in \cite{M1402959}, and convergence order of the approximated spectrum is quadratic as is theoretically expected (Theorem~\ref{cotadoblepandeo}). 

\begin{table}[H]
\begin{center}
\caption{Test 2: Computed lowest eigenvalues $\l_{h}^{(i)}$, $1\le
i\le4$ of the unit circular plate on $\cT_h^2$ .}
\begin{tabular}{|c||c||ccc||c||c||c|c|c}
\hline
$\CT_h$   & $\l_{h}^{(i)}$  & $N=16$ & $N=32$ & $N=64$ &
Order & Extr.&\cite{MR2473688} \\
\hline
\hline
& $\l_{h}^{(1)}$  & 14.9103  & 14.7353  & 14.6947 &   2.11  & 14.6825&14.6834\\
& $\l_{h}^{(2)}$   & 26.9907 &  26.5212 &  26.4096 &   2.07 &  26.3746&26.3784\\
$\CT_h^2$ & $\l_{h}^{(3)}$   & 27.0417  & 26.5302 &  26.4127  &  2.12  & 26.3775&26.3786\\
& $\l_{h}^{(4)}$   & 42.0813  & 41.0416 &  40.7888   & 2.04  & 40.7076&40.7143\\
& $\l_{h}^{(5)}$   & 42.20075  & 41.0648 &  40.7945   & 2.07  & 40.7099&40.7160\\
\hline
 \end{tabular}
\label{TABLA:clamped}
\end{center}
\end{table}

\subsection{Effects of the stabilization in the computed spectrum}

Apart from the previous discussion, here, we focus to analyze the influence of the stabilizer, i.e., $S^K(\cdot,\cdot)$ which is defined in Section~\ref{subsec:disc_b_f} on discrete eigenvalues and optimal convergence rate associated with the mesh refinement. Further, it is well known that in eigenvalue problems, when numerical methods depend on some
stabilization parameter, spurious eigenvalues, and suboptimal order of convergence of the eigenvalues may appear and it is important to 
find a threshold in which the stabilization parameter does not introduce this pollution.
In this direction, we consider the Stokes eigenvalue problem with mixed boundary conditions as follows:
 Find $\lambda\in\mathbb{R}$, the velocity $\bu$, and the pressure $p$ such that
\begin{equation*}\label{def:stokes_eigen_mixed_boundary}
\left\{
\begin{array}{rcll}
-\nu\Delta\bu +\nabla p & = & \lambda\bu&  \text{ in } \quad \Omega, \\
\div\bu & = & 0 & \text{ in } \quad \Omega, \\
\bu & = & \mathbf{0} & \text{ on } \quad \Gamma_D,\\
(\nabla\bu-p\mathbb{I})\boldsymbol{n} & = & \mathbf{0} & \text{ on } \quad \Gamma_N,
\end{array}
\right.
\end{equation*}
where the boundary $\partial\O$  admits a disjoint partition $\partial\O:=\Gamma_D\cup\Gamma_N$, where both, the Dirichlet and Neumann parts of the boundary, have positive measure,  $\boldsymbol{n}$ is the outward normal vector to $\partial\O$, and $\mathbb{I}\in\mathbb{R}^{2\times 2}$ is the identity matrix.  For this test, we will consider $\O:=(0,1)^2$ as computational domain, whereas the bottom fo the square will be clamped (i.e. $\bu=\boldsymbol{0}$ on $\Gamma_D$) and the rest of the sides of the square will be free, imposing the Neumann boundary condition on them.
 
On the other hand, to perform this test we redefine the stabilizer as $\alpha S^K(\cdot,\cdot) $, where $\alpha >0$. In Table~\ref{TABLA:mixesp}, we posted the lowest eigenvalues computed by our scheme for different choice of $\alpha$ with square mesh ($\CT_h^1$) and voronoi mesh ($\CT_h^2$), and refinement level $N=11$. Let us remark that this choice of $N$ is to observe correctly the eigenfunctions in order to determine if we are in presence of spurious or physical eigenvalues.  The numbers inside boxes represent the localized spurious eigenvalues on the computed spectrum.
\begin{table}[H]
\begin{center}
\caption{Test~3: Computed lowest eigenvalues for different values of $\alpha$  with $N=11$ and $\CT_{h}^{1}$,$\CT_{h}^{2}$, and $\CT_{h}^{5}$.}
\begin{tabular}{|l|l|l|l|l|l|l|l|l|l|}
\hline
\multicolumn{5}{ |c| }{$\CT_{h}^{1}$} \\
\hline
\hline
  $\alpha=\frac{1}{10}$ & $\alpha=\frac{1}{5}$ &  $\alpha=1$ & $\alpha=5$ & $\alpha=10$ \\
\hline
2.4716  &  2.4716  &  2.4716   & 2.4716  &  2.4716\\
    3.3520 &   5.3533  &  6.0680  &  6.3638 &   6.8070\\
    5.0031 &  14.1244 &  15.0430&   15.6086 &  17.0387\\
   \fbox{5.3806} &  16.0821 &  22.5483&   22.5483 &  22.5483\\
   \fbox{5.7303}&   \fbox{20.1632} &  25.2649 &  27.7358 &  31.4481\\
    5.7810&   \fbox{20.8073} &  41.7462 &  44.9704&   48.4881\\
    5.9995 &  \fbox{21.1000}&   43.6492&   50.9110 &  64.3476\\
    \fbox{6.1079} &  22.5483 &  55.9620 &  64.3476 &  65.3221\\
    \fbox{6.2193} &  22.6441 &  64.3476 &  68.8324 &  84.0526\\
    \fbox{6.4613} &  \fbox{22.7580}  & 66.8821&   81.8611 & 106.0351\\
\hline
\hline
\multicolumn{5}{ |c| }{$\CT_{h}^{2}$} \\
\hline
\hline
  $\alpha=\frac{1}{10}$ & $\alpha=\frac{1}{5}$ &  $\alpha=1$ & $\alpha=5$ & $\alpha=10$ \\
\hline
 2.4375&    2.4628&    2.4687 &   2.4734 &   2.4862\\
    3.8019&    5.6057&    6.1282&    6.3671 &   6.7384\\
    \fbox{5.2752}  & 13.8688&   15.0188 &  15.5603&   16.6973\\
    \fbox{5.4793} &  18.5939  & 22.2737 &  22.6543&   23.6218\\
    \fbox{5.5345} &  20.2889&   25.8953 &  27.6940 &  30.6311\\
    \fbox{5.6137} &  \fbox{20.8660}  & 42.9686 &  45.0023 &  49.5569\\
    \fbox{5.6353}   &\fbox{20.9625} &  43.4884&   50.5334 &  60.8378\\
    \fbox{5.6817} &  \fbox{21.5805} &  59.3004 &  64.9147&   72.6551\\
    \fbox{5.7266}   &\fbox{22.0167} &  61.1749 &  68.7617 &  83.6547\\
    \fbox{5.7487} &  \fbox{22.1313} &  67.8385 &  80.2792 & 101.6238\\
    \hline
    \hline
\multicolumn{5}{ |c| }{$\CT_{h}^{5}$} \\
\hline
\hline
  $\alpha=\frac{1}{10}$ & $\alpha=\frac{1}{5}$ &  $\alpha=1$ & $\alpha=5$ & $\alpha=10$ \\
\hline
\hline
2.4695   & 2.4708   & 2.4714    &2.4718 &   2.4729\\
    3.5024&    5.4102 &   6.0747 &   6.3770  &  6.8564\\
   \fbox{ 5.3396}&   14.1771&   15.0476 &  15.6272&   17.1093\\
    \fbox{5.6891}&   \fbox{16.6210} &  22.5287 &  22.5693&   22.6544\\
    \fbox{5.7402}&   \fbox{20.7998} &  25.4045 &  27.7871&   31.5454\\
    \fbox{5.9391}&   \fbox{21.5002}&   42.1502 &  45.1100 &  48.8237\\
    \fbox{6.1818}&   \fbox{21.8507} &  43.6642 &  50.9626 &  64.7043\\
    \fbox{6.2279} &  \fbox{22.3891} &  56.3702 &  64.4149 &  65.6116\\
    \fbox{6.4287} &  \fbox{23.1814} &  64.1430 &  69.0110  & 84.9012\\
    \fbox{6.5270} &  \fbox{23.2058} &  67.3464 &  81.9551 & 108.0546\\
\hline
\end{tabular}
\label{TABLA:mixesp}
\end{center}
\end{table}

From Table \ref{TABLA:mixesp} we observe the presence of spurious eigenvalues for $\alpha= 1/16$ and $\alpha=1/4$ for the considered meshes, whereas for $\alpha\geq 1$ we observe that the lowest computed eigenvalues are correct approximations of the physical ones, letting us to infer that a safe threshold  of $\alpha$ for the approximation of the spectrum is from one and forward. It is important that this threshold has been identify for this configuration of the problem (namely, the geometrical domain, boundary conditions, physical paraneters) and may change on other configuration.

Eventually, we investigate the convergence order of the spectrum for different choice of parameters $\alpha$, particularly for  $\alpha \in \{4^{-2}, 4^{-1}, 1, 4, 4^2\} $. We observed from Table~\ref{TABLA:mix:oc} that when the parameter $\alpha<1$, the order cf convergence is affected and is not optimal. On the other hand, for $\alpha\geq 1$ the order fo convergence is improved and the optimal order is attained. This is expectable since we have already see that the safe threshold to recover the physical spectrum is precisely for $\alpha\geq 1$.

\begin{table}[H]
\begin{center}
\caption{Test~3: Computed lowest eigenvalues $\l_{h}^{(i)}$, $1\le
i\le5$ for different choice of $\alpha$ on square domain with mixed boundary condition, and $\CT_h^1$.}
\begin{tabular}{|c||c||ccc||c||c||c|c|c }
\hline
$\alpha$   & $\l_{h}^{(i)}$  & $N=32$ & $N=64$ & $N=128$ &
Order & Extr.&\cite{LEPE2021113753} \\
\hline
\hline
& $\l_{h}^{(1)}$  & 2.4956  & 2.4809  & 2.4740 &   1.09  & 2.4678 &2.4674\\
& $\l_{h}^{(2)}$   & 6.1066 &  6.1922 &  6.2353 &   0.99 &  6.2790&6.2791\\
$\frac{1}{16}$ & $\l_{h}^{(3)}$   & 15.0555  & 15.1287 &  15.1679  & 0.90  & 15.2132&15.2096\\
& $\l_{h}^{(4)}$   & 22.4876  & 22.3358 &  22.2680   & 1.16  & 22.2136 &22.2064\\
& $\l_{h}^{(5)}$   & 26.3993  & 26.7144 &  26.8400   & 1.33  & 26.9227 &26.9525\\
\hline
\hline
& $\l_{h}^{(1)}$ &  2.4686   &2.4679 &  2.4677   & 1.81  & 2.4676&2.4674\\
& $\l_{h}^{(2)}$  &  6.3145  & 6.2877 &  6.2795   & 1.71  & 6.2758 &6.2791\\
$\frac{1}{4}$& $\l_{h}^{(3)}$     &15.3759  & 15.2522   &15.2129   & 1.66  & 15.1948 &15.2096\\
& $\l_{h}^{(4)}$    & 22.2470   &22.2136 &  22.1984   & 1.15 & 22.1861 &22.2064\\
& $\l_{h}^{(5)}$    & 27.2197   &27.0025 &  26.9401   & 1.80 & 26.9149 &26.9525\\
  \hline
  \hline
& $\l_{h}^{(1)}$ &  2.4692   &2.4678  & 2.4675 &   2.22  & 2.4674 &2.4674\\
& $\l_{h}^{(2)}$  &  6.4223  & 6.3147&   6.2891 &   2.07 &  6.2810 &6.2791\\
$1$& $\l_{h}^{(3)}$     & 15.8558 &  15.3767 &  15.2532 &   1.96  & 15.2107&15.2096\\
& $\l_{h}^{(4)}$   & 22.3720  & 22.2479 &  22.2165  &  1.98  & 22.2057&22.2064\\
& $\l_{h}^{(5)}$   & 28.0639  & 27.2199 &  27.0087  &  2.00  & 26.9384 &26.9525\\
\hline
\hline
& $\l_{h}^{(1)}$  & 2.4691   &2.4679 &  2.4676   & 2.00 & 2.4675 &2.4674\\
& $\l_{h}^{(2)}$   &  6.3144  & 6.2882 & 6.2812  &  1.90 &  6.2786&6.2791\\
$4$ & $\l_{h}^{(3)}$  &  15.3762  & 15.2528 &  15.2212  &  1.97  & 15.2105 &15.2096\\
 & $\l_{h}^{(4)}$  &  22.2476  & 22.2162 &  22.2089  &  2.11  & 22.2067 &22.2064\\
 & $\l_{h}^{(5)}$  &  27.2201  & 27.0148 &  26.9591  &  1.89  & 26.9387 &26.9525\\
\hline
\hline
& $\l_{h}^{(1)}$  & 2.4693   &2.4679 &  2.4676   & 2.22 & 2.4675 &2.4674\\
& $\l_{h}^{(2)}$   &  6.3146  & 6.2891 & 6.2821  &  1.87 &  6.2795&6.2791\\
$16$ & $\l_{h}^{(3)}$  &  15.3765  & 15.2532 &  15.2219  &  1.98  & 15.2113 &15.2096\\
 & $\l_{h}^{(4)}$  &  22.2483  & 22.2167 &  22.2100  &  2.24  & 22.2082 &22.2064\\
 & $\l_{h}^{(5)}$  &  27.2210  & 27.0148 &  26.9600  &  1.91  & 26.9401 &26.9525\\
\hline
 \end{tabular}
\label{TABLA:mix:oc}
\end{center}
\end{table}


\subsection{A posteriori test}
We end this section by reporting a numerical test that allows us to  assess the performance  of the a posteriori estimator 
defined in \eqref{eq:global_eta}. With this aim, we have implemented in a MATLAB code a lowest order VEM scheme
on arbitrary polygonal meshes. In order to apply the adaptive virtual  element method, we shall generate a sequence of nested conforming meshes using the loop
\begin{center}
	\Large \textrm{solve $\rightarrow$ estimate $\rightarrow$ mark $\rightarrow$ refine.} 
\end{center}
This  mesh refinement algorithm described in \cite{MR3342219},  consists in the splitting of each element of the mesh into $n$ quadrilaterals ($n$ being the number of edges of the polygon) by connecting the barycenter of the element with the midpoint of each edge, which will be named as  \textbf{Adaptive VEM}. Notice that although this process is initiated with a mesh of triangles, the successively created meshes will contain other kind of convex polygons, as it can be seen from Figure \ref{FIG:refined}. Both schemes are based on the strategy of refining those elements $\E\in\CT_h$ that satisfy
 $$ \eta_{\E}\geq 0.5 \max_{{\E'\in\CT_{h}}}\{\eta_{\E'}\}.$$

In Figure \ref{FIG:strategy} we present an example of the VEM strategy for the refinement of certain polygons.

 \begin{figure}[H]
	\begin{center}
		\begin{minipage}{13cm}
                          \centering
                          \hspace{.5 cm}
                           \centering
 \includegraphics[height=3.7cm, width=3.7cm,angle=180]{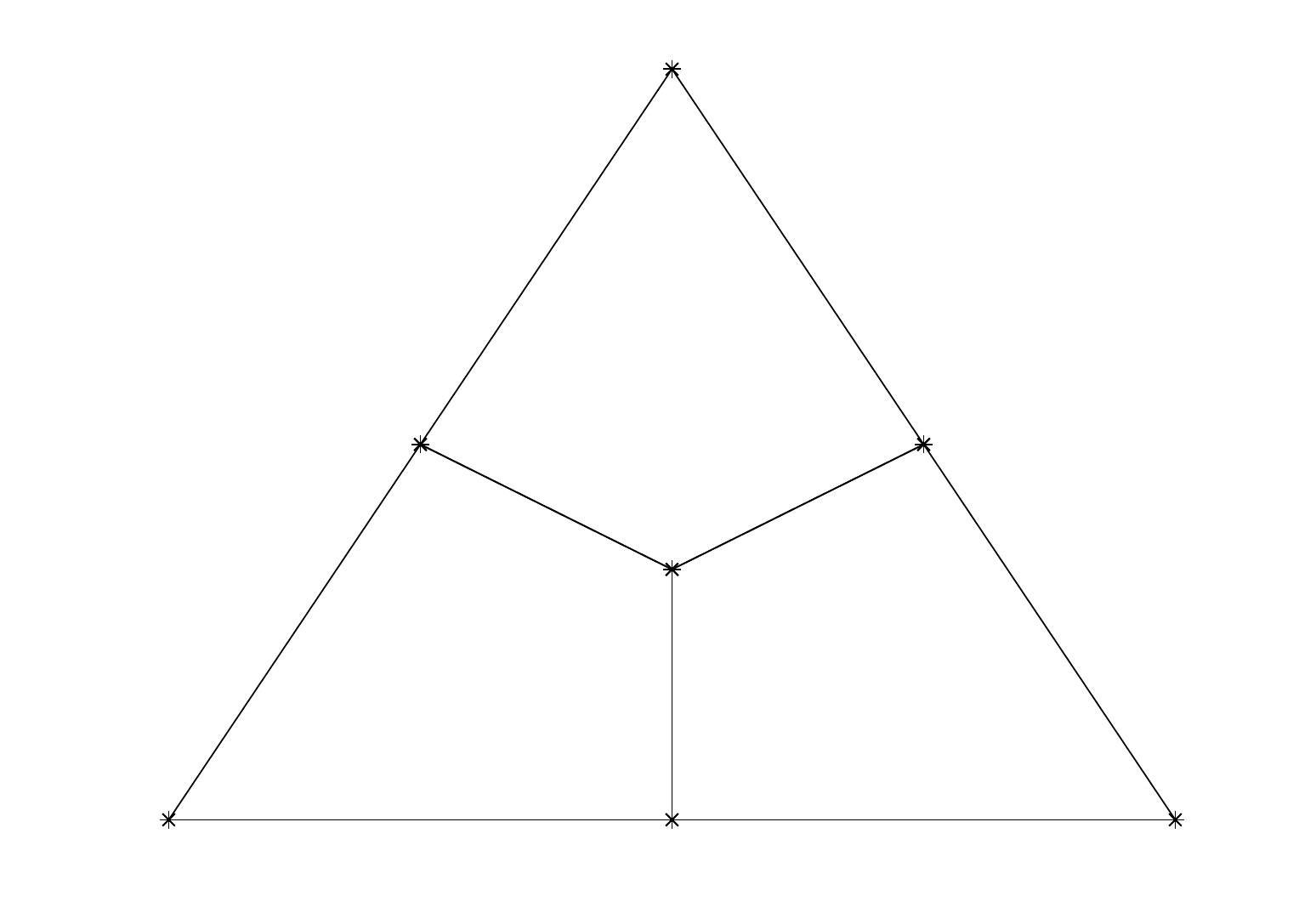}
                           \hspace{.5 cm}
                   \end{minipage}
		\caption{Example of refined elements for the VEM strategy.}
		\label{FIG:strategy}
	\end{center}
\end{figure}
 \subsection{L-shaped Domain} 
The computational domain for this test is $\O=(-1,1)\times (-1,1)\setminus (-1,0)\times (-1,0)$
 and the only boundary condition is $\bu = \0$. For the L-shaped domain, the re-entrant angle leads to a lack of regularity for some eigenfunctions, as has been studied in \cite{MR2473688}, the convergence rates of the errors for the eigenvalues, vary between $1.7 \leq r \leq2$, depending on the regularity of the corresponding eigenfunction. Our goal is to recover the optimal convergence order for the lowest eigenvalue with the proposed estimator. We note that since the true eigenvalue is not known, we have chosen the extrapolated value $\l_1 = 32.13183$ as the exact solution, which is in line with what is presented in the literature (see \cite{M1402959,LEPE2023114798,MR2473688}). 
 Figure\ref{FIG:refined} reports the adaptively refined meshes obtained with  VEM procedure for different  initial meshes.
 \begin{figure}[H]
	\begin{center}
		\begin{minipage}{15cm}
			\centering\includegraphics[height=4.7cm, width=4.7cm]{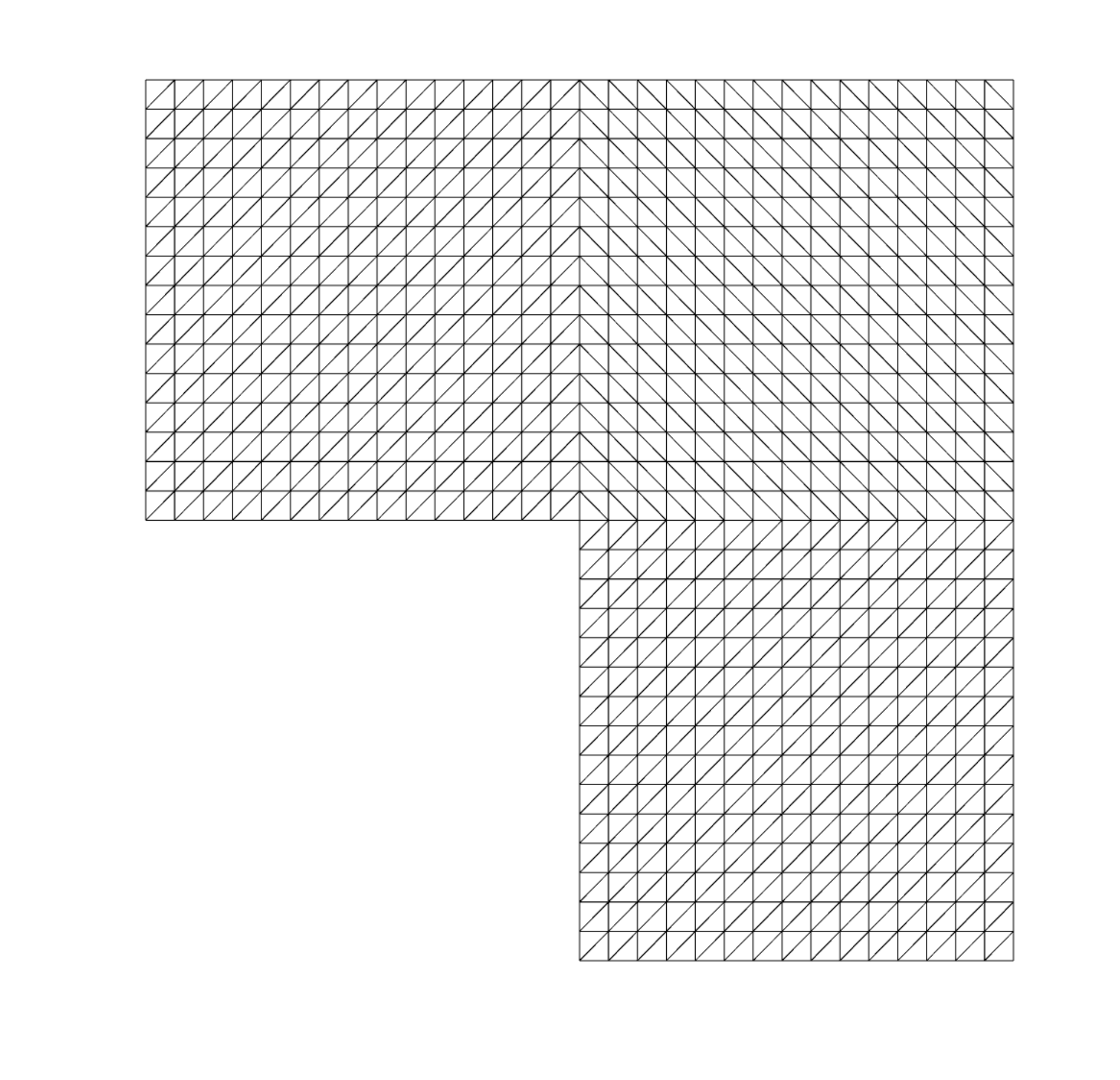}
			\centering\includegraphics[height=4.7cm, width=4.7cm]{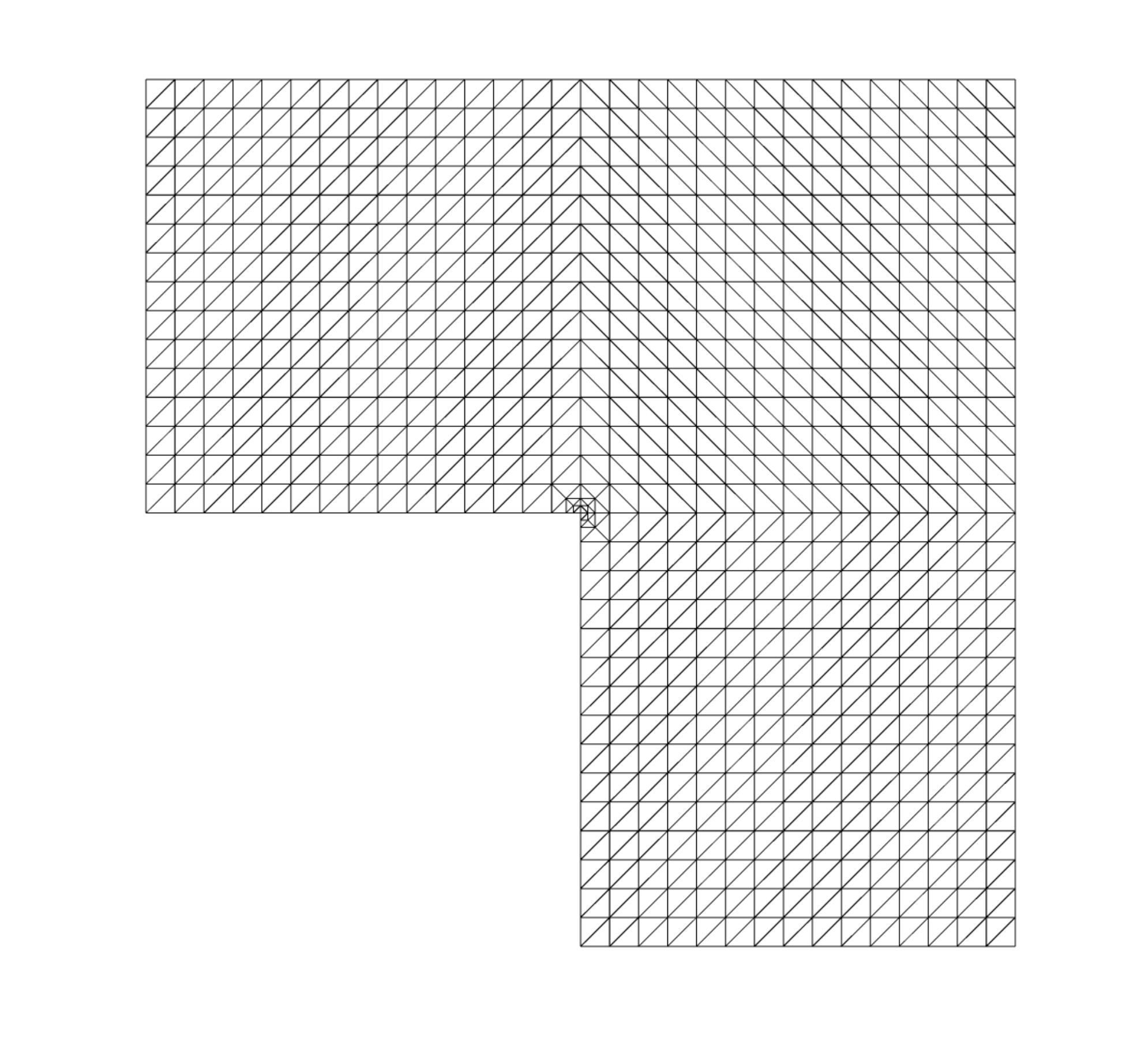}
                         \centering\includegraphics[height=4.7cm, width=4.7cm]{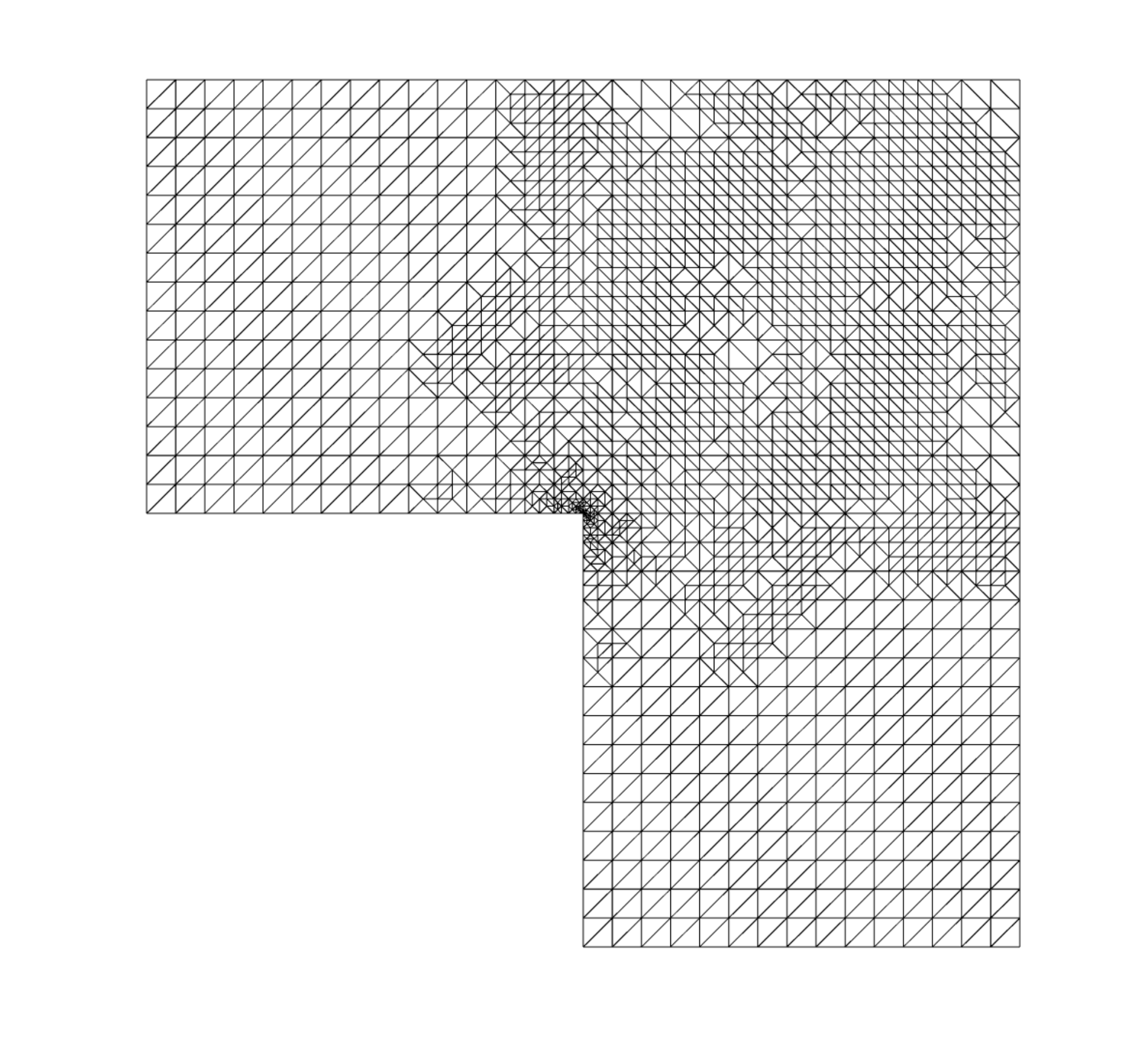}\\
                         \centering\includegraphics[height=4.7cm, width=4.7cm]{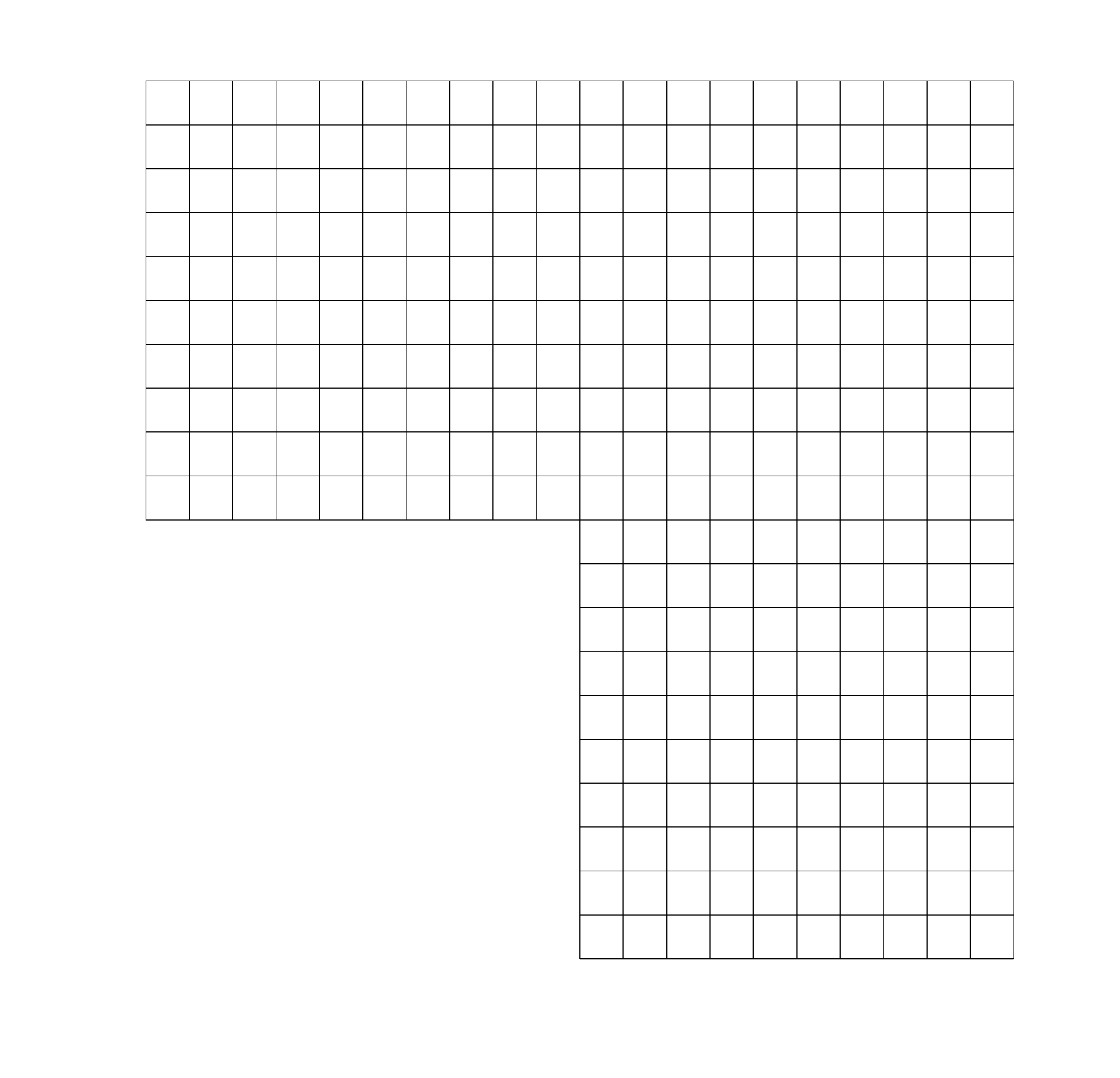}
                          \centering\includegraphics[height=4.7cm, width=4.7cm]{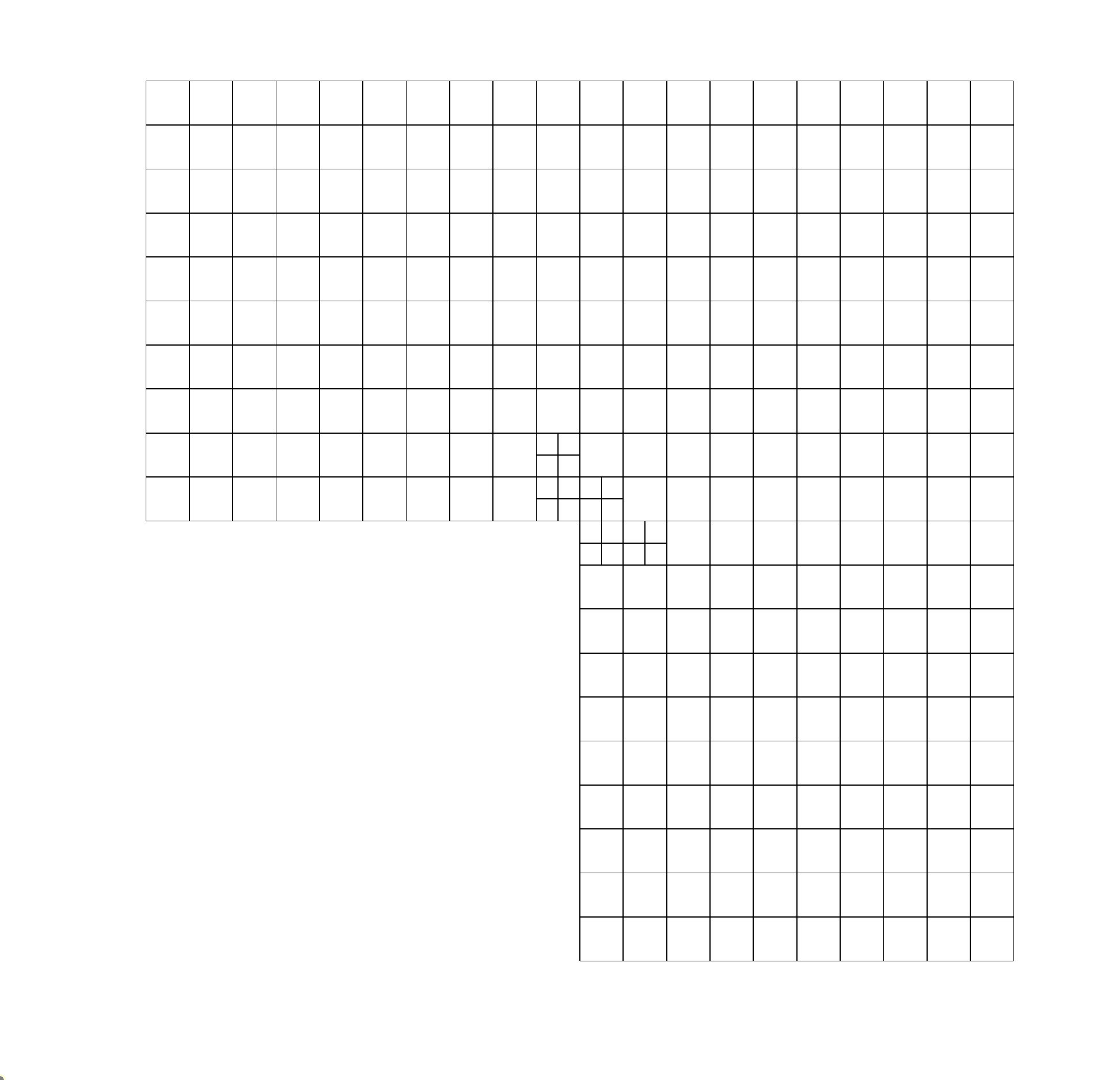}
                          \centering\includegraphics[height=4.7cm, width=4.7cm]{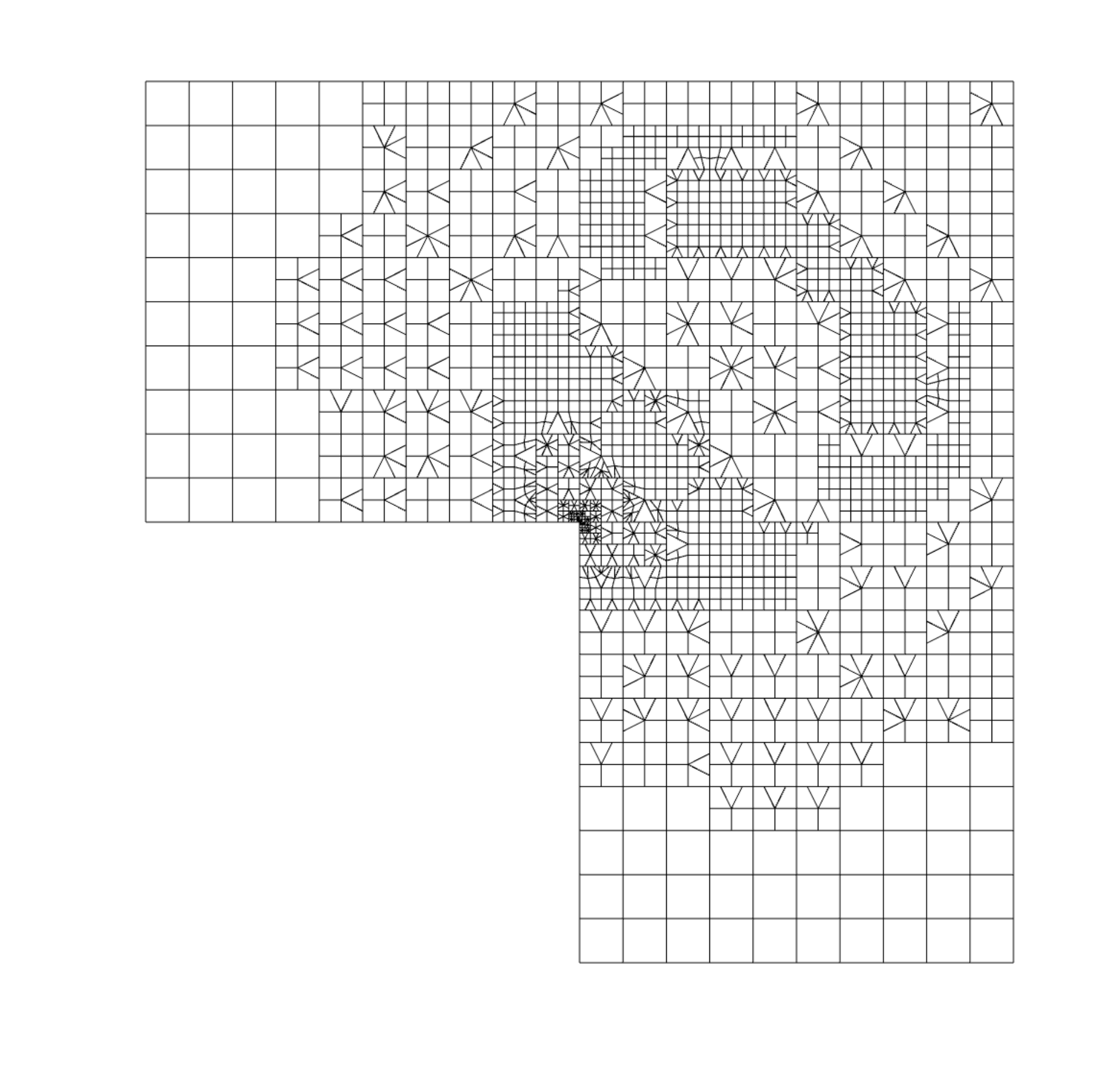}\\
                            \centering\includegraphics[height=4.7cm, width=4.7cm]{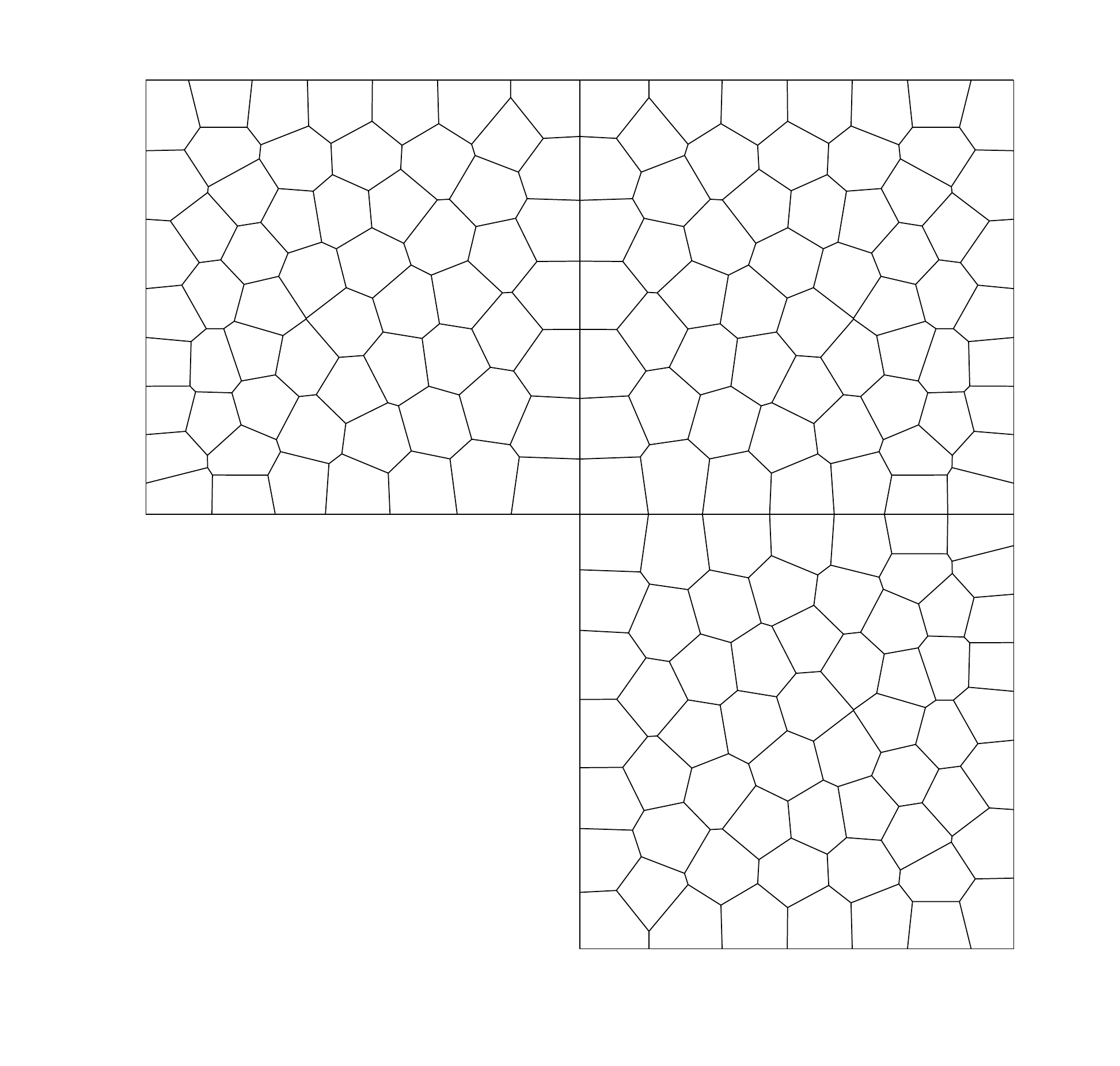}
                          \centering\includegraphics[height=4.7cm, width=4.7cm]{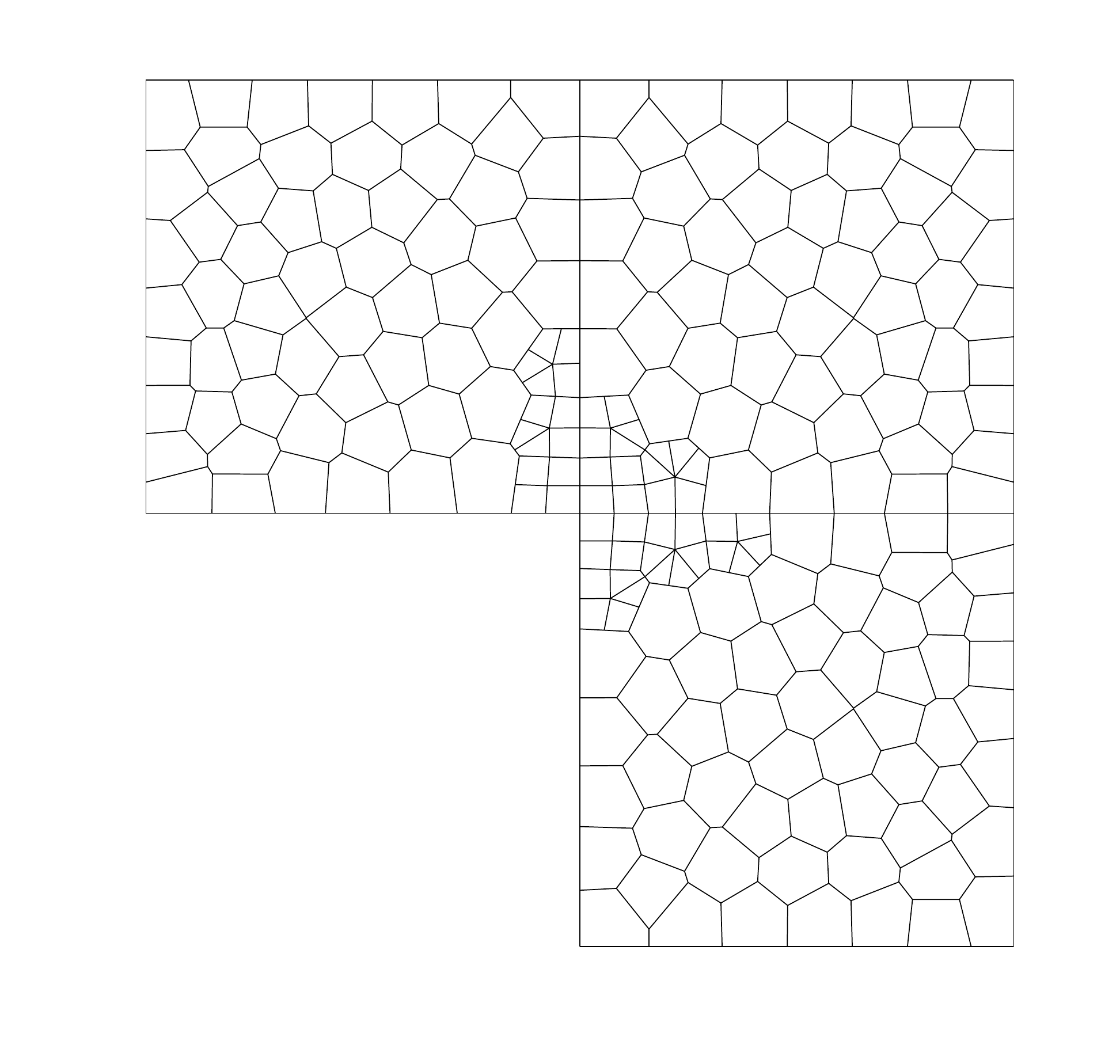}
                          \centering\includegraphics[height=4.7cm, width=4.7cm]{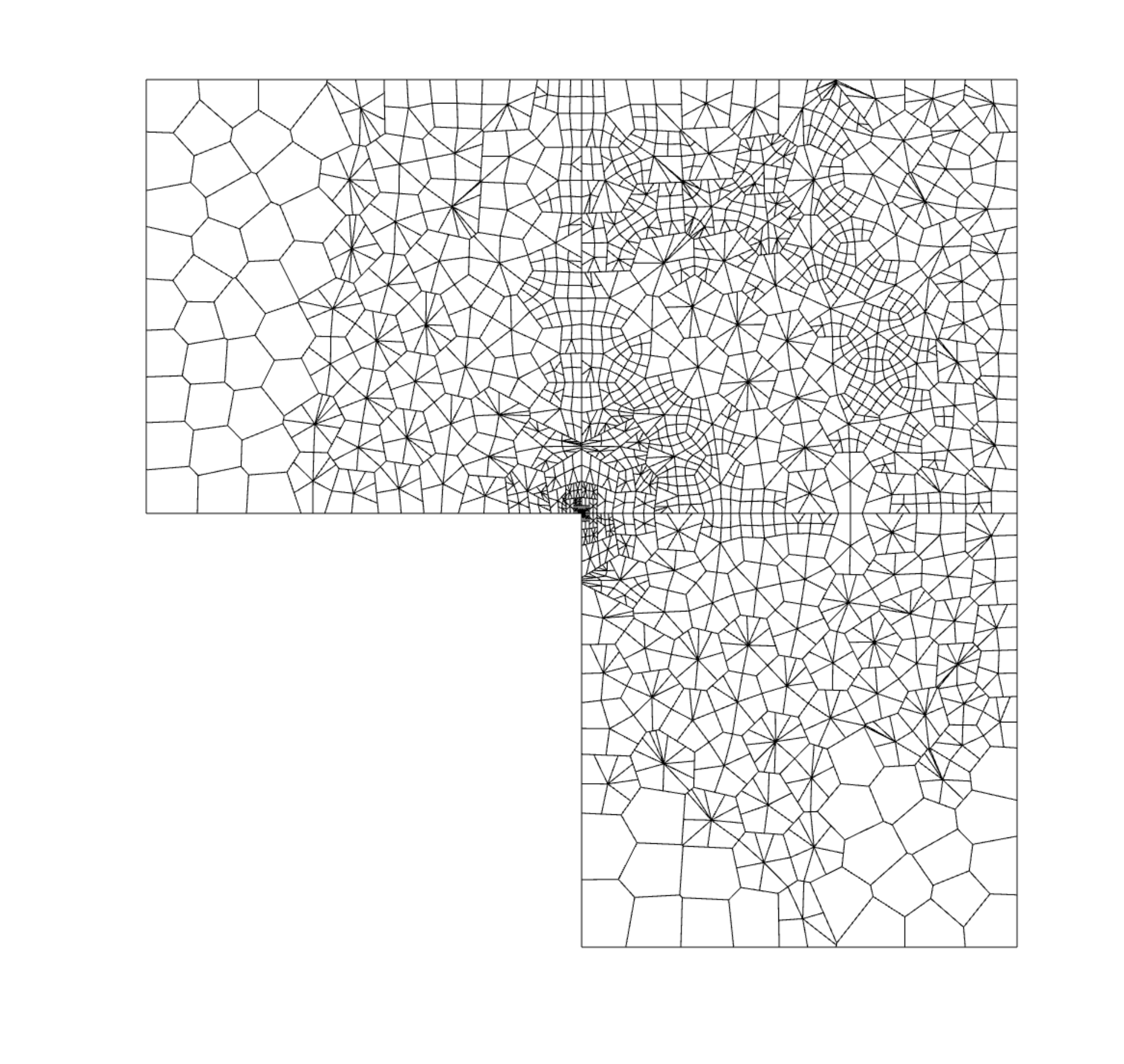}
                   \end{minipage}
		\caption{Adaptively refined meshes obtained with the VEM scheme at refinement steps 0, 1 and 10.}
		\label{FIG:refined}
	\end{center}
\end{figure}
We report in Table \ref{TABLA:1} the lowest eigenvalue $\l_{h}^{(1)}$ on uniformly refined meshes, adaptively refined meshes with VEM schemes.
 \begin{table}[H]
\begin{center}
\caption{Test 4. Lowest eigenvalue   $\l_{h}$ computed with different schemes.}
\begin{tabular}{|c|c||c|c||c|c||c|c||c|c}
  \hline
    \multicolumn{2}{|c||}{Uniform VEM} & \multicolumn{2}{c||}{Adaptive VEM triangles}& \multicolumn{2}{c||}{Adaptive VEM squares}& \multicolumn{2}{c||}{Adaptive VEM voronoi}\\
    \hline
     $N$ & $\l_{h}^{(1)}$  &   $N$ & $\l_{h}^{(1)}$  &   $N$ & $\l_{h}^{(1)}$&   $N$ & $\l_{h}^{(1)}$  \\
\hline
1622   & 32.7016 &3557   & 32.4306   &1322 &   30.6989&1637 &  30.6918 \\
     6242   & 32.3157 &3590&    32.4377 & 1400   & 31.0615&  1877&   30.9977\\
    24482  &  32.1893 & 3623  &  32.4325 & 2166   & 31.5866&  3597  & 31.7255\\
    96962 &   32.1514 & 3786   & 32.4235& 3326   & 31.6915  &  4583  & 31.9341\\
& & 3949   & 32.4122   & 4644   & 31.9244&  5783  & 32.0543\\
&  & 5658   & 32.2994    & 5900   & 32.0087&   6992  & 32.1015 \\
   &   & 8208   & 32.2271     & 9998   & 32.0627&   10928  & 32.1013\\
 &   & 9699   & 32.2184 & 15490  &  32.0930&   17022  & 32.1182\\
  & & 13549  &  32.1986 & 21032  &  32.1087&   23567  & 32.1257\\
  & & 20204  &  32.1728& 29484  &  32.1175&   33731  & 32.1284\\
 &  & 32487  &  32.1582& 47610  &  32.1246&  51612  & 32.1302\\
 &  & 42232  &  32.1528& & & &\\
     \hline 
     Order   &$\mathcal{O}\left(N^{-0.83}\right)$&       Order   & $\mathcal{O}\left(N^{-1.11}\right)$&       Order   & $\mathcal{O}\left(N^{-1.44}\right)$&       Order   & $\mathcal{O}\left(N^{-1.96}\right)$ \\
       \hline
       $\l_1$  &32.1321&  $\l_{1}$  &32.1321 &  $\l_{1}$  &32.1321&  $\l_{1}$  &32.1321\\
     \hline
    \end{tabular}
\label{TABLA:1}
\end{center}
\end{table}  
\noindent We observe from  Figure \ref{FIG:errorL} that the refinement schemes lead to the correct convergence rate. Moreover, the performance of the adaptive VEM is slightly better than that of the adaptive with triangles, as it admits dangling nodes, which implies that it refines fewer elements and as the number of elements is part of the degrees of freedom of the scheme, this adaptor achieves better convergence.
\begin{figure}[H]
	\begin{center}
		\begin{minipage}{13cm}
			\centering\includegraphics[height=8.5cm, width=10cm]{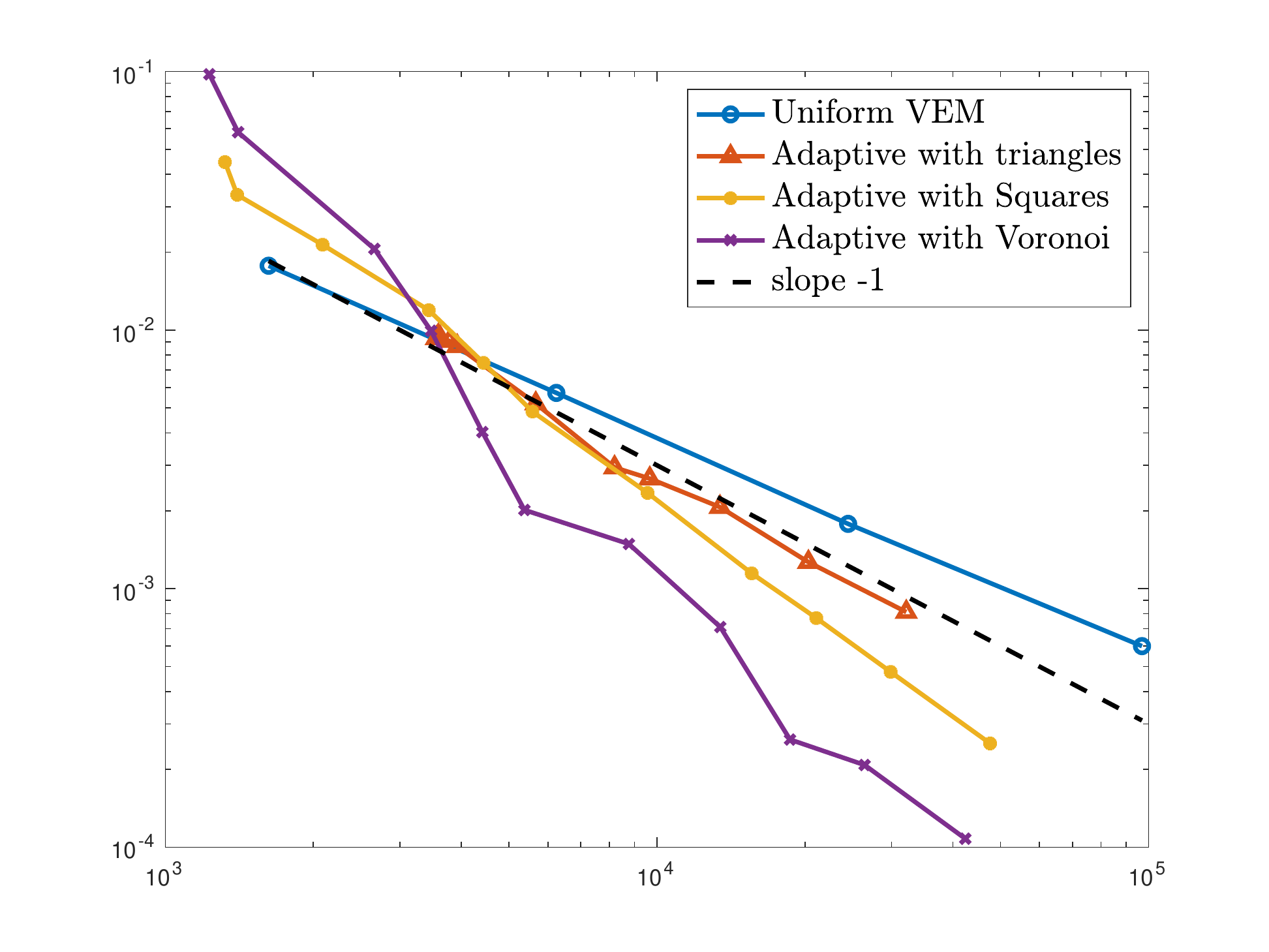}
                   \end{minipage}
		\caption{Comparison between error.}
		\label{FIG:errorL}
	\end{center}
\end{figure}
\noindent We report in Table \ref{TABLA:3} the error $\texttt{err}(\lambda_1):=|\l_{1}-\l_{h}^{(1)}|/\l_{1}$ and the estimators $\boldsymbol{\eta}^2$ at each step of the adaptive
VEM scheme. We include in the table the  terms  $\displaystyle \boldsymbol{\theta}^{2}:=\sum_{\E\in\CT_{h}}\boldsymbol{\theta}_{\E}^{2}$, $\displaystyle \boldsymbol{R}^{2}:=\sum_{\E\in\CT_{h}}\boldsymbol{R}_{\E}^{2},$
which appears from the inconsistency of the VEM, and 
$\boldsymbol{J}_h:=\displaystyle\sum_{\E\in\CT_{h}}\left(\sum_{\ell\in\mathcal{E}_\E} h_{\E}\|\boldsymbol{J}_{\ell}\|_{0,\ell}^{2}\right),$
which arises from the edge residuals. We also report in the aforementioned table the effectivity indexes $\texttt{eff}(\boldsymbol{\eta}):=\texttt{err}(\lambda_1)/\boldsymbol{\eta}^{2}$.

\begin{table}[H]
\begin{center}
\caption{Test 4: Components of the error estimator and effectivity indexes on the adaptively refined meshes for an initial mesh of squares.}
\vspace{0.3cm}
\begin{tabular}{|c||c||c||c||c||c||c||c||c|}
\hline
$N$   &   $\texttt{err}(\lambda_1)$   & $\boldsymbol{\theta}^{2}$ &$\boldsymbol{R}^{2}$ & $\boldsymbol{J}_h^{2}$ &  $\boldsymbol{\eta}^2$ & $\texttt{eff}(\boldsymbol{\eta})$ \\
\hline
1322  &44.6038e-3   &  4.3986   &  9.6635e-3   &  3.6216  &   8.0298   &  5.5548e-3\\
1400   &  33.3180e-3 &    3.5460&     9.5581e-3 &    3.5874  &   7.1430&     4.6644e-3\\
 2166  & 16.9778e-3  &   1.9308&     6.1570e-3 &    2.5228  &   4.4597  &   3.8069e-3\\
  3326  & 13.7109e-3 &    1.2730 &    2.9580e-3 &    1.4304  &   2.7063  &   5.0663e-3\\
4644      & 6.4637e-3 &  703.2267e-3 &    2.3145e-3  &   1.1400  &   1.8456  &   3.5023e-3\\
5900      & 3.8404e-3&   494.9957e-3 &    2.1043e-3 &  991.0927e-3 &    1.4882 &    2.5806e-3\\
9998      & 2.1593e-3 &  293.7349e-3 &    1.1748e-3 &  597.6668e-3 &  892.5765e-3  &   2.4192e-3\\
15490    & 1.2163e-3&   179.8121e-3 &  695.9132e-6 &  375.3807e-3   &555.8887e-3  &   2.1880e-3\\
21032  & 729.0094e-6 &  124.4353e-3 &  554.9716e-6 &  289.8030e-3 &  414.7933e-3 &    1.7575e-3\\
29484  & 453.2853e-6 &   89.6844e-3  & 415.6462e-6 &  219.0012e-3 &  309.1012e-3  &   1.4665e-3\\
47610  & 232.4108e-6 &   57.4341e-3 &  239.8696e-6 &  135.6897e-3  & 193.3636e-3  &   1.2019e-3\\\hline
  \end{tabular}
\label{TABLA:3}
\end{center}
\end{table}

\noindent It can be seen from the Table \ref{TABLA:3} that the effectivity indexes are bounded above and below far from 0 and the inconsistency and edge residual terms are roughly speaking of the same order, none of them being asymptotically negliglible.

\bibliographystyle{siam}
\footnotesize
\bibliography{bib_LOQ}

\end{document}